\documentclass{amsart}
\usepackage{amsmath,amssymb}
\usepackage{graphicx}
\usepackage[shortlabels]{enumitem}
\usepackage{multirow}
\usepackage{hyperref}
\usepackage{url}
\usepackage{longtable}
\usepackage{tikz}
\newtheorem{theorem}{Theorem}[section]

\newtheorem{proposition}[theorem]{Proposition}
\newtheorem{lemma}[theorem]{Lemma}
\newtheorem{corollary}[theorem]{Corollary}

\newtheorem{remark}[theorem]{Remark}

\def\tr{\mathop{\rm tr }\nolimits}

\title[The characteristic polynomial of a Seidel matrix]%
{On equiangular lines in $17$ dimensions and \\ the characteristic polynomial of a Seidel matrix}

\author{Gary R. W. Greaves}
\address{School of Physical and Mathematical Sciences, \\
Nanyang Technological University, \\
21 Nanyang Link, Singapore 637371, Singapore}
\email{grwgrvs@gmail.com}
\thanks{The first author was supported by the Singapore
Ministry of Education Academic Research Fund (Tier 1);  grant number:  RG127/16.}

\author{Pavlo Yatsyna}
\address{Department of Mathematics, \\
 Royal Holloway, University of London, \\
  Egham Hill, Egham, Surrey, TW20 0EX, UK}
  \email{pvyatsyna@gmail.com}
  \thanks{}


\begin{document}
	
\maketitle

\begin{abstract}
	For $e$ a positive integer, we find restrictions modulo $2^e$ on the coefficients of the characteristic polynomial $\chi_S(x)$ of a Seidel matrix $S$.
	We show that, for a Seidel matrix of order $n$ even (resp.\ odd), there are at most $2^{\binom{e-2}{2}}$ (resp.\ $2^{\binom{e-2}{2}+1}$) possibilities for the congruence class of $\chi_S(x)$ modulo $2^e\mathbb Z[x]$.
	% For a fixed natural number $e$, we observe that these upper bounds on the number of congruence classes are achieved for $n$ large enough.
	% We show that, for a Seidel matrix of order $n$ even (resp.\ odd), there are at most $2^{\binom{e-2}{2}}$ (resp.\ $2^{\binom{e-2}{2}+1}$) possibilities for $\chi_S(x)$ over the ring $\mathbb Z_{2^e}[x]$.
	As an application of these results we obtain an improvement to the upper bound for the number of equiangular lines in $\mathbb R^{17}$, that is, we reduce the known upper bound from $50$ to $49$.
\end{abstract}

\section{Introduction}

For a matrix $M$, we define the characteristic polynomial of $M$ as $\chi_M(x) = \det(xI-M)$.
A symmetric $\{0,\pm 1\}$-matrix having $0$ on the diagonal and all other entries equal to $\pm 1$ is called a \textbf{Seidel matrix}.
Much like adjacency matrices, Seidel matrices describe the adjacency of a graph.
Indeed, a Seidel matrix $S$ of order $n$ corresponds to a graph on $n$ vertices where vertex $i$ is adjacent to vertex $j$ if and only if $S_{ij} = -1$.
We call this graph the \textbf{underlying graph} of $S$.
Throughout, $\mathbf{1}$ denotes the all-ones vector and $J = \mathbf{1} \mathbf{1}^\top$ denotes the all-ones matrix, of size appropriate to the context in which it is used.
Let $A$ be the adjacency matrix of the underlying graph of a Seidel matrix $S$.
Then the matrices $A$ and $S$ are related by the equation $S = J-I-2A$.

Both the matrices $S$ and $A$ are symmetric matrices with integer entries.
Therefore, the eigenvalues of both are necessarily real and also algebraic integers (they are the zeros of the characteristic polynomial, which is monic with integer coefficients).
However, the question of which eigenvalues that satisfy these necessary conditions can be eigenvalues of these matrices have very different answers.

We call an algebraic integer \textbf{totally real} if all its (Galois) conjugates are real.
Hoffman~\cite{hoffman75} conjectured and Estes~\cite{estes92} showed that (also see Salez~\cite{salez15}) every totally real algebraic integer is an eigenvalue of an adjacency matrix of a graph.
In other words, for all totally real algebraic integers $\alpha$, there exists a graph with adjacency matrix $A$ such that the minimal polynomial of $\alpha$ divides the characteristic polynomial $\chi_A(x)$ of $A$.
On the other hand, Haemers (see \cite{GG18}) showed that the characteristic polynomial of a Seidel matrix is congruent to $\chi_{J-I}(x)$ modulo $2\mathbb Z[x]$.
Suppose $J$ is of order $n$, then $\chi_{J-I}(x) = (x-n+1)(x+1)^{n-1}$.
Hence, modulo $2 \mathbb Z[x]$, we have $\chi_{J-I}(x) \equiv (x+1)^{n}$ if $n$ is even and $\chi_{J-I}(x) \equiv x(x+1)^{n-1}$ if $n$ is odd.
Observe that, since $\mathbb Z[x]/2\mathbb Z[x]$ is a UFD, the polynomial $x^2-2$ cannot be a factor of a Seidel matrix.
Hence, in particular, the totally real algebraic integer $\sqrt{2}$ cannot be an eigenvalue of any Seidel matrix.%\footnote{We first heard this pointed out by Jack Koolen (personal communication).} . 

We further develop this modular restriction on the coefficients of the characteristic polynomial of a Seidel matrix that was observed by Haemers.
Our main contribution is to characterise the congruence classes of the characteristic polynomial of a Seidel matrix modulo $2^e \mathbb Z[x]$ for $e$ a positive integer.
% Further, we apply our results to the problem that popularised Seidel matrices.
Further, we give an improvement to the upper bound for the number of equiangular lines in $\mathbb R^{17}$.
% \begin{proposition}\label{pro:evenmod2}
% 	Let $S$ be a Seidel matrix of order $n$ even.
% 	Then $\chi_{S}(x) \equiv (x+1)^n \mod {2\mathbb Z[x]}$.
% \end{proposition}

Seidel matrices were introduced by Van Lint and Seidel~\cite{vLintSeidel66} as a tool for studying systems of lines in Euclidean space, any two of which make the same angle.
Such a system of lines is called an \textbf{equiangular} line system.
A classical problem from 1948~\cite{haantjes48} is, given a positive integer $d$, to find the maximum cardinality $N(d)$ of an equiangular line system in $\mathbb R^d$.
This problem has recently been enjoying lots of attention, for example, the last two years have produced the publications \cite{balla18,bukh16,GG18,GKMS16,KingTang,LinYu,okudayu16,SzOs18,Yu17,waldron}.
 % and the complex analogue of this problem is of great interest to the quantum information community.
However, the value of $N(d)$ remains unknown for $d$ as small as $14$.
% As an application of this characterisation, we give an improved upper bound for equiangular lines systems in $\mathbb R^{17}$.
% This improvement is not a straightforward application of our main results; to obtain the improved upper bound we also make use the theory of graph angles and totally positive algebraic integers of small trace.
One contribution of this paper is to provide an improvement to the upper bound of $N(17)$, that is, to show that $N(17) \leqslant 49$ (see Remark~\ref{rem:1749}).

In Table~\ref{tab:equi} below, we give the currently known (including the improvement from this paper) values or lower and upper bounds for $N(d)$ for $d$ at most $23$.

\begin{table}[htbp]
	\begin{center}
		\setlength{\tabcolsep}{2pt}
	\begin{tabular}{c|cccccccccccccccc}
		$d$  & 2 & 3        & 4           & 5  & 6  & 7--13 & 14  & 15 & 16 & 17 & 18 & 19 & 20 & 21 & 22 & 23\\\hline
		$N(d)$  & 3 & 6        & 6           & 10 & 16 & 28    & 28--29  & 36 & 40--41 & 48--49 & 56--60 & 72--75 & 90--95  & 126 & 176 & 276
	\end{tabular}
	\end{center}
	\caption{Bounds for the sequence $N(d)$ for $2\leqslant d\leqslant 23$.  A single number is given in the cases where the exact number is known.  The improvement from this paper, $N(17) \leqslant 49$, is included.  A word of caution: some tables in the literature present the lower bounds for $N(d)$ as if they were the exact value of $N(d)$.}
	\label{tab:equi}
\end{table}

The plan of the paper is as follows.
In Section~\ref{sec:a_congruence_for_the_trace_of_powers_of_an_adjacency_matrix}, we establish congruences for trace powers of the adjacency matrix of a graph.
Next, in Section~\ref{sec:a_modular_characterisation_of_the_characteristic_polynomial_of_a_seidel_matrix}, we apply use these congruences to establish modular restrictions on the coefficients of the characteristic polynomial of a Seidel matrix.
In Section~\ref{sec:50lines}, we turn our attention to necessary conditions for the existence of a system of $50$ equiangular lines in $\mathbb R^{17}$.
We establish connections between the existence of equiangular line systems and totally positive algebraic integers with small trace.
In Section~\ref{sec:nonexist}, we show that the necessary conditions from Section~\ref{sec:50lines} cannot be satisfied.
Finally, we conclude by remarking on how the approach presented in this paper can be applied to $N(d)$ for $d \ne 17$.

\section{A congruence for the trace of powers of the adjacency matrix of a graph} % (fold)
\label{sec:a_congruence_for_the_trace_of_powers_of_an_adjacency_matrix}

% section a_congruence_for_the_trace_of_powers_of_an_adjacency_matrix (end)

% \subsection{Counting orbits of closed walks} % (fold)
% \label{sec:counting_orbits}

In this section we apply Burnside's lemma to find a congruence modulo $2N$ for a weighted sum of traces of powers of a graph-adjacency matrix.

Let $\Gamma$ be a graph and let $\mathbf{x}$ be a closed walk of length $N$ in $\Gamma$; we write $\mathbf x = x_0x_1\dots x_{N-1}$ where $x_i$ is adjacent to $x_{i+1}$ for each $i \in \{0,\dots,N-1\}$ with indices reduced modulo $N$.
There is a natural correspondence between the vertices of the closed walk $\mathbf{x}$ and the vertices of a regular $N$-gon.
Hence, under this correspondence, we consider the dihedral group $D_N$ of order $2N$ acting on the set of closed $N$-walks of $\Gamma$.
Let $N \geqslant 3$ and write $D_N = \langle r, s \; | \; r^N, s^2, (rs)^2  \rangle$.
For $g \in D_N$, we denote by $\operatorname{fix}_\Gamma(g)$ the set of closed $N$-walks of $\Gamma$ fixed by $g$.

\begin{lemma}\label{lem:fix}
	Let $\Gamma$ be a graph with adjacency matrix $A$ and let $N \geqslant 3$.
	% Let $D_N = \langle r, f \; | \; r^N, f^2, (rf)^2  \rangle$ act on the set of closed $N$-walks of $\Gamma$.
	Then
	\begin{enumerate}[(i)]
		\item $|\operatorname{fix}_\Gamma(r^k)| = \tr (A^{\operatorname{gcd}(k,N)})$, for all $k \in \mathbb Z$;
		\item $|\operatorname{fix}_\Gamma(r^{2k}s)| = 0$, for all $k \in \mathbb Z$;
		\item $\displaystyle
		|\operatorname{fix}_\Gamma(r^{2k+1}s)| = 
		\begin{cases}
			\mathbf{1}^\top A^{N/2} \mathbf{1}, & \text{ if $N$ is even}  \\ 
			0, & \text{ if $N$ is odd}  \\ 
		\end{cases}
		\quad$ (for all $k \in \mathbb Z$).	
	\end{enumerate}
\end{lemma}
\begin{proof}
	Let $\mathbf x = x_0x_1\dots x_{N-1}$ be a closed $N$-walk that is fixed by some element $g \in D_N$.
	Observe that, if $c$ is a cycle of the group element $g$ then, for each $i$ and $j$ in $c$, we have $x_i = x_j$.
	
	First suppose that $g = r^k$ for some $k \in \mathbb Z$.
	Then $g$ has order $m=N/\operatorname{gcd}(k,N)$.
	Therefore, $g$ consists of $N/m$ cycles each of length $m$.
	It follows that, for all $i \in \{0,\dots,N/m-1\}$, we have $x_{im}x_{im+1}\dots x_{im+m-1} = x_0x_1\dots x_{m-1}$.
	Hence $\operatorname{fix}_\Gamma(g)$ consists of closed $N/m$-walks repeated $m$ times.
	Since $\tr (A^{N/m})$ is equal to the number of closed walks of length $N/m = \operatorname{gcd}(k,N)$, we have established (i).
	
	Now suppose $N$ is odd and $g = r^{k}s$ for some $k \in \mathbb Z$.
	In this case, $g$ consists of $\lfloor N/2 \rfloor$ cycles of length $2$.
	It follows that one pair of adjacent vertices of $\mathbf{x}$ must be equal, but there are no such closed walks since $\Gamma$ has no loops.
	Whence we have (ii) and (iii) for $N$ odd.  

	In the remainder of the proof we assume that $N$ is even.
	Suppose that $g = r^{2k}s$ for some $k \in \mathbb Z$.
	In this case, $g$ consists of $N/2$ cycles of length $2$.
	Then two pairs of adjacent vertices of $\mathbf{x}$ must be equal, but there are no such closed walks since $\Gamma$ has no loops.
	This yields (ii).
	
	Finally, suppose that $g =r^{2k+1}s$ for some $k \in \mathbb Z$.
	In this case, $g$ consists of $N/2-1$ cycles of length $2$.
	Without loss of generality, we can assume that $x_0$ and $x_{N/2}$ are fixed by $g$.
	Then, for each $i \in \{1,\dots,N/2-1\}$, we must have $x_i = x_{N-i}$.
	Hence $\operatorname{fix}_\Gamma(g)$ consists of closed $N$-walks made up of an $N/2$-walk together with its inverse.
\end{proof}

For a positive integer $a$, we use $\varphi(a)$ to denote Euler's totient function of $a$.
% Now we can prove the main result of this section.
The following result is the main result of this section, which follows from Lemma~\ref{lem:fix} via a straightforward application of Burnside's lemma (that $\frac{1}{|D_N|}\sum_{g \in D_N} |\operatorname{fix}_\Gamma(g)|$ is an integer).
% Then the next result is a straightforward application of Burnside's lemma

\begin{lemma}[{cf.\ \cite[Corollary 5a]{Harary}}]\label{lem:trrelmod}
	Let $\Gamma$ be a graph with adjacency matrix $A$ and let $N \geqslant 4$ be an even integer.
 Then
	\[
	\sum_{d \;|\; N} \varphi(N/d) \tr(A^d) + \frac{N}{2} \mathbf{1}^\top A^{N/2} \mathbf{1} \equiv 0 \mod{2N}.
\]
\end{lemma}
Note that we also have a similar congruence when $N$ is odd, which also follows from Lemma~\ref{lem:fix}.
% the proof is similar to the proof of Lemma~\ref{lem:trrelmod}.

\begin{lemma}[{cf.\ \cite[Corollary 5a]{Harary}}]\label{lem:trrelmodOdd}
	Let $\Gamma$ be a graph with adjacency matrix $A$ and let $N \geqslant 3$ be an odd integer.
 Then
	\[
	\sum_{d \;|\; N} \varphi(N/d) \tr(A^d) \equiv 0 \mod{2N}.
\]
\end{lemma}

% For a graph $\Gamma$ we write $w_p = \mathbf{1}^\top A^{p} \mathbf{1}$ for the number of walks of length $p$ in $\Gamma$.
% Write $w_l(v)$ for the number of walks of length $l$ starting at the vertex $v$.
A graph $\Gamma$ is called an \textbf{Euler graph} if the degree of each of its vertices is even.
Let $S$ be a Seidel matrix.
We define the \textbf{switching class} of $S$ to be the set of underlying graphs of Seidel matrices of the form $DSD$ where $D$ is a diagonal matrix with diagonal entries equal to $\pm 1$.

Later we will need the following result due to Seidel~\cite{Hageetal03,Seidel74}.
\begin{theorem}\label{thm:OddNEuler}
	Let $S$ be a Seidel matrix of order $n$ odd.
	Then there exists a unique Euler graph in the switching class of $S$.
\end{theorem}

The next result is a standard result from linear algebra.

% \begin{lemma}\label{lem:jaj}
% 	Let $A$ be a symmetric integer matrix of order $n$ whose diagonal entries are all even and let $\mathbf{v} \in \mathbb Z^n$.
% 	Then $\mathbf{v}^\top A^i \mathbf{v}$ is even for all integers $i \geqslant 1$.
% \end{lemma}
\begin{lemma}\label{lem:jaj}
	Let $A$ be a symmetric integer matrix whose diagonal entries are all even. 
	Then $\mathbf{1}^\top A^i \mathbf{1}$ is even for all integers $i \geqslant 1$.
\end{lemma}

We now establish congruences for the number of walks of a given length in an Euler graph.

% \begin{lemma}
% 	\label{lem:wp4}
%     Let $\Gamma$ be an Euler graph with adjacency matrix $A$.
%     Then $\mathbf{1}^\top A \mathbf{1} \equiv 0 \mod 2$, $\mathbf{1}^\top A^{2} \mathbf{1} \equiv 0 \mod 4$, and $\mathbf{1}^\top A^{i} \mathbf{1} \equiv 0 \mod 8$ for all $i \geqslant 3$.
% \end{lemma}
\begin{lemma}
	\label{lem:wp4}
    Let $\Gamma$ be an Euler graph with adjacency matrix $A$.
    Then $\mathbf{1}^\top A \mathbf{1} \equiv 0 \mod 2$ and $\mathbf{1}^\top A^{i} \mathbf{1} \equiv 0 \mod 4$ for all $i \geqslant 2$.
\end{lemma}
% \begin{proof}
%     Write $w_l(v)$ for the number of walks of length $l$ starting at the vertex $v$.
%     Then
%     \[
%         w_2(v) = \sum_{w \sim v} d_w,
%     \]
%     which is an even integer.
%
%     Now we assume that $w_l(v)$ is even for all vertices $v$ for all $l \leqslant n$.
%     We have that $w_{n+1}(v) = \sum_{w \sim v} w_n(w)$, which is clearly even.
%     By induction, $w_l(v)$ is even for all vertices and for all $l \geqslant 2$.
%     The result follows since $w_p = \sum_v d_v w_{p-1}(v)$.
% \end{proof}
% \begin{proof}
%     By Lemma~\ref{lem:jaj}, we have that $\mathbf{1}^\top A \mathbf{1}$ is even.
% 	Since $\Gamma$ is an Euler graph, we can write $A \mathbf{1} = 2 \mathbf{v}$ for some integer vector $\mathbf{v}$.
% 	Hence $\mathbf{1}^\top A^{2} \mathbf{1} = 4 \mathbf{v}^\top \mathbf{v} \equiv 0 \mod 4$.
% 	Finally, for $i \geqslant 3$, we have $\mathbf{1}^\top A^{i} \mathbf{1} = 4\mathbf{v}^\top A^{i-2} \mathbf{v}$, which is divisible by $8$, by Lemma~\ref{lem:jaj}.
% \end{proof}
\begin{proof}
    By Lemma~\ref{lem:jaj}, we have that $\mathbf{1}^\top A \mathbf{1}$ is even.
	Since $\Gamma$ is an Euler graph, we can write $A \mathbf{1} = 2 \mathbf{v}$ for some integer vector $\mathbf{v}$.
	Hence, for $i \geqslant 2$, we have $\mathbf{1}^\top A^{i} \mathbf{1} = 4\mathbf{v}^\top A^{i-2} \mathbf{v}$, which is divisible by $4$.
\end{proof}

Using  Lemma~\ref{lem:trrelmod} together with Lemma~\ref{lem:wp4}, we obtain the following lemma.

\begin{lemma}\label{lem:trrelEuler}
	Let $\Gamma$ be an Euler graph with adjacency matrix $A$ and let $N \geqslant 4$ be an even integer.
	Then
	\[
		\tr(A^N) \equiv -\sum_{\substack{d \;|\; N \\ d \ne N}} \varphi(N/d) \tr(A^d) \mod{2N}.
	\]
\end{lemma}

% section counting_orbits (end)

\section{Relations for the coefficients of characteristic polynomials modulo powers of $2$} % (fold)
\label{sec:a_modular_characterisation_of_the_characteristic_polynomial_of_a_seidel_matrix}

\subsection{A relation for characteristic polynomials} % (fold)
\label{sub:a_characteristic_polynomial_relation}

In this section we establish a relation between the characteristic polynomial of a Seidel matrix $S$ and the characteristic polynomial of a graph in the switching class of $S$. 
If $\Gamma$ is a graph with adjacency matrix $A$ then its Seidel matrix has the form $S = J-I-2A$.
The characteristic polynomial $\chi_S(x)$ of $S$ can be written as $\chi_S(x) = \chi_{J-2A}(x+1)$.
With this in mind, we instead consider the relation between $\chi_A(x)$ and $\chi_{J-2A}(x)$.

\begin{lemma}\label{lem:coefficientMap}
	Let $A$ be a matrix of order $n$.
	Write $\chi_{J-2A}(x) = \sum_{i=0}^n a_i x^{n-i}$ and $\chi_{A}(x) = \sum_{i=0}^n b_i x^{n-i}$.
	Then
	\[
		a_r = (-2)^r \left(b_r + \frac{1}{2} \sum_{i=1}^{r} b_{r-i} \, \mathbf{1}^\top A^{i-1}  \mathbf{1} \right ).
	\]
\end{lemma}
\begin{proof}
	By the matrix determinant lemma,
	\[
		\chi_{J-2A}(x) = \chi_{-2A}(x) - \mathbf{1}^\top \operatorname{adj}(xI+2A)\mathbf{1}.
	\]
	Write $\chi_{-2A}(x) = \sum_{i=0}^n c_i x^{n-i}$.
	The adjugate matrix can be written~\cite[p.\ 38]{FDC55} as
	\[
		\operatorname{adj}(xI+2A) = \sum_{i=0}^{n-1} (-2A)^{n-1-i} \sum_{j=0}^i x^{i-j}c_{j}.
	\]
	Note that we have $c_i = (-2)^i b_i$ for all $i \in \{0,\dots,n\}$.
	%, using the fact that $\mathbf{j}^\top A^r  \mathbf{j} = \tr(JA^r)$, 
	The result then follows by equating coefficients.
\end{proof}

Now we record a couple of corollaries to Lemma~\ref{lem:coefficientMap}.
First, a surprisingly strong restriction on $\chi_{J-2A}(x)$ where $A$ is the adjacency matrix of a graph of order $n$ even.

\begin{corollary}\label{cor:evencond}
	Let $A$ be the adjacency matrix of a graph of order $n$ even and write $\chi_{J-2A}(x) = \sum_{i=0}^n a_i x^{n-i}$.
	Then $2^r$ divides $a_r$ for all $r \in \{0,\dots,n\}$.
\end{corollary}
\begin{proof}
	By Lemma~\ref{lem:coefficientMap}, it suffices to show that $\mathbf{1}^\top A^{i-1}  \mathbf{1}$ is even for all $i \geqslant 1$.
	By Lemma~\ref{lem:jaj}, for all $i \geqslant 2$, we have that $\mathbf{1}^\top A^{i-1}  \mathbf{1}$ is even, and, for $i=1$, we have $\mathbf{1}^\top A^{i-1}  \mathbf{1} = n$, which is also even.
\end{proof}

\begin{remark}
	\label{rem:a1a2}
	Let $A$ be an adjacency matrix of a graph of order $n$.
	It is clear that the trace of $J-2A$ equals $n$ and the trace of $(J-2A)^2$ equals $n^2$.
	Write $\chi_{J-2A}(x) = \sum_{i=0}^n a_i x^{n-i}$.
	Obviously $a_0 = 1$.
	Furthermore, using Newton's identities, we see that $a_1 = -n$ and $a_2 = 0$.
\end{remark}

Denote by $\mathcal C_n$ the set of all Seidel matrices of order $n$.
Given a positive integer $e$, define the set $\mathcal P_{n,e} = \{ \chi_S(x) \mod 2^e \mathbb Z[x] \; | \; S \in \mathcal C_n \}$.
Using Remark~\ref{rem:a1a2} together with Corollary~\ref{cor:evencond} allows us to obtain the following.

\begin{corollary}\label{cor:countCharPolySeidelEven}
	Let $n$ be an even integer and $e$ be a positive integer.
	Then the cardinality of $\mathcal P_{n,e}$ is at most $2^{\binom{e-2}{2}}$.
\end{corollary}

Clearly, if $n$ is small compared to $e$ then the cardinality of $\mathcal P_{n,e}$ will be strictly less than $2^{\binom{e-2}{2}}$.  
Indeed, for $n = 2$ the cardinality of $\mathcal P_{n,e}$ is $1$ for all $e$.
% However, it is straightforward to check that the bound in Corollary~\ref{cor:countCharPolySeidelEven} is sharp for small values of $e$ ($e \leqslant 7$) and large enough even $n$.
However, it is straightforward to check that, for small values of $e$ ($e \leqslant 7$), there exist even $n$ giving equality in the bound in Corollary~\ref{cor:countCharPolySeidelEven}.
% However, for a fixed positive integer $e$, we have empirical evidence that suggests the bound in Corollary~\ref{cor:countCharPolySeidelEven} is sharp for large enough even $n$.
We conjecture that, for all $e \in \mathbb N$, there exists $N \in \mathbb N$ such that $|\mathcal P_{n,e}| = 2^{\binom{e-2}{2}}$ for all even $n \geqslant N$.

By Corollary~\ref{cor:countCharPolySeidelEven}, there is only one congruence class modulo $2\mathbb Z[x]$ for characteristic polynomials of Seidel matrices of even order.
Since $J-I$ is a Seidel matrix with $\chi_{J-I}(x) = (x-(n-1))(x+1)^{n-1}$, we note the following corollary.

\begin{corollary}[cf.\ {\cite[Lemma 2.2]{GG18}}]\label{cor:evenmod2}
	Let $S$ be a Seidel matrix of order $n$ even.
	Then $\chi_S(x) \equiv (x+1)^n \mod 2\mathbb Z[x]$.
\end{corollary}

Empirically, we observed that the analogous bound for $\mathcal P_{n,e}$ when $n$ is odd should be $2^{\binom{e-2}{2}+1}$.
We continue in pursuit of this bound, which culminates in Corollary~\ref{cor:countCharPolySeidelOdd}.

Note the following well-known result from linear algebra.

\begin{lemma}\label{lem:odddet}
	Let $A$ be a symmetric integer matrix of order $n$ odd whose diagonal entries are all even.
	Then $\det A$ is even.
\end{lemma}

We will need a corollary of this lemma.

\begin{corollary}\label{cor:evenbr}
	Let $A$ be a symmetric integer matrix of order $n$ whose diagonal entries are all zero and write $\chi_{A}(x) = \sum_{i=0}^n b_i x^{n-i}$.
	Then $b_r$ is even for all odd $r$.
\end{corollary}

We now establish a result similar to Corollary~\ref{cor:evencond}.

\begin{lemma}\label{lem:oddcond1}
	Let $A$ be an adjacency matrix of a graph of order $n$ and write $\chi_{J-2A}(x) = \sum_{i=0}^n a_i x^{n-i}$.
	Then $2^r$ divides $a_r$ for all $r$ even.
\end{lemma}
\begin{proof}
	Let $\chi_{A}(x) = \sum_{i=0}^n b_i x^{n-i}$.
	Using Lemma~\ref{lem:coefficientMap}, it suffices to show that $\sum_{i=1}^{r} b_{r-i} \, \mathbf{1}^\top A^{i-1}  \mathbf{1}$ is even.
	By Lemma~\ref{lem:jaj}, for all $i \geqslant 2$, we have that $\mathbf{1}^\top A^{i-1}  \mathbf{1}$ is even.
	% Since $b_{k}$ is equal to plus or minus the sum of all principal $k$-minors of $A$, by Lemma~\ref{lem:odddet}, the coefficient $b_{k}$ is even for all odd $k$.
	And, by Corollary~\ref{cor:evenbr}, the coefficient $b_{r-1}$ is even for all $r$ even.
\end{proof}

\subsection{An application to Euler graphs} % (fold)
\label{sub:applications_to_euler_graphs}

% subsection applications_to_euler_graphs (end)

Let $S$ be a Seidel matrix of order $n$ odd.
By Theorem~\ref{thm:OddNEuler}, we may assume that the underlying graph of $S$ is an Euler graph.
We therefore focus on Euler graphs $\Gamma$.

\begin{corollary}\label{cor:btoa}
	Let $\Gamma$ be an Euler graph of order $n$ odd and let $A$ be its adjacency matrix.
	Write $\chi_{J-2A}(x) = \sum_{i=0}^n a_i x^{n-i}$ and $\chi_{A}(x) = \sum_{i=0}^n b_i x^{n-i}$.
	Then, for all $r \in \{1,\dots, \frac{n-1}{2} \}$, we have
	\begin{align*}
		b_{2r} &\equiv \frac{-a_{2r+1}}{2^{2r} n} \mod {4}; \\
		%b_{2r-1} &\equiv \frac{a_{2r+1}+na_{2r}+a_{2r-1}a_3}{2^{2r-1}}   \mod {8}.
		b_{2r-1} &\equiv \frac{a_{2r+1}+a_{2r}+a_{2r-1}a_3}{2^{2r-1}}   \mod {4}.
	\end{align*}
\end{corollary}
\begin{proof}
	Using Lemma~\ref{lem:wp4}, Lemma~\ref{lem:coefficientMap}, and Corollary~\ref{cor:evenbr}, we have that 
	\begin{align}
		\frac{a_{2r+1}}{2^{2r-1}} &\equiv -2nb_{2r} \mod {8}; \label{eqn:e1} \\
		\frac{na_{2r}}{2^{2r-1}} &\equiv 2nb_{2r} + b_{2r-1}n^2 + nb_{2r-2} \mathbf{1}^\top A\mathbf{1} \mod {4}; \label{eqn:e2} \\
		a_3 &\equiv -4b_2n \mod 8.  \label{eqn:a3}
	\end{align}
	Furthermore, since $b_1 = 0$, we can write $a_2 = 4b_2 + 2\mathbf{1}^\top A\mathbf{1}$.
	By Remark~\ref{rem:a1a2}, we have $a_2 = 0$, then using \eqref{eqn:a3}, we find that $a_3 \equiv 2n\mathbf{1}^\top A\mathbf{1} \equiv 2\mathbf{1}^\top A\mathbf{1} \mod 8$.
	Hence, 
	\[
		\frac{a_{2r-1}a_3}{2^{2r-1}} \equiv -nb_{2r-2} \mathbf{1}^\top A\mathbf{1} \mod 4.
	\]
	Since, by Lemma~\ref{lem:oddcond1}, $2^r$ divides $a_{2r}$, using \eqref{eqn:e1} and \eqref{eqn:e2}, we can write 
	$$b_{2r-1} \equiv \frac{a_{2r+1}+a_{2r}+a_{2r-1}a_3}{2^{2r-1}}   \mod {4},$$
	as required.
	%
	% \begin{align*}
	% 	b_{2k} &\equiv \frac{-a_{2k+1}}{2^{2k} n} \mod {2^{2}}; \\
	% 	b_{2k-1} &\equiv \frac{a_{2k+1}+na_{2k}+a_{2k-1}a_3}{2^{2k-1}}  \mod {2^{2}}.
	% \end{align*}
	% The result follows since $b_{2k-1}$ is congruent to $0$ or $2$ modulo $4$.
\end{proof}

For a nonzero integer $a$, the $2$\textbf{-adic valuation} $\nu_2(a)$ of $a$ is the multiplicity of $2$ in the prime factorisation of $a$.
The $2$-adic valuation of $0$ is defined to be $\infty$.
For an integer $b$ relatively prime to $a$, the $2$-adic valuation $\nu_2(a/b)$ of $a/b$ is defined as $-\nu_2(b)$ if $b$ is even and $\nu_2(a)$ otherwise.
Denote by $B_2(a)$ the number of $1$s in the binary expansion of $a$.

\begin{lemma}\label{lem:powersmod}
	Let $s$ and $t$ be integers satisfying $s \equiv t \mod 4$.
	Then for all positive integers $m$, we have $s^m \equiv t^m \mod 2^{\nu_2(m)+2}$.
\end{lemma}
\begin{proof}
	Write $s = t + 4u$ for some integer $u$.
	Since $s^m = t^m + \sum_{i=1}^m \binom{m}{i}(4u)^i t^{m-i}$, it suffices to show that
	\begin{align}
		\nu_2 \left (\binom{m}{i} \right )+2i &\geqslant \nu_2(m)+2, \quad \text{ for all $i \in \{1,\dots,m\}$}. \label{ineq:nu}
	\end{align}
	By a theorem of Kummer~\cite{Kummer52}, we have $\nu_2(\binom{m}{i}) = B_2(m-i)+B_2(i)-B_2(m)$.
	Hence, \eqref{ineq:nu} becomes $B_2(m-1) \leqslant 2i-3+B_2(m-i)+B_2(i)$ for all $i \in \{1,\dots,m\}$.
	
	Observe that the number of $1$s in the binary expansion of $m-i$ is at least $B_2(m-1)-B_2(i-1)$.
	Thus, we have $B_2(m-1) \leqslant B_2(m-i)+B_2(i-1)$.
	Finally, it is straightforward to verify the inequality $B_2(i-1) \leqslant 2i-3+B_2(i)$ for all $i \in \{1,\dots,m\}$.
\end{proof}

We will also need the next lemma in preparation for the subsequent result.

\begin{lemma}\label{lem:multinomial2val}
	Let $l$ be a positive integer and let $m_1, m_2, \dots, m_{l}$ be nonnegative integers having a positive sum.
	Let $m \in \{m_i \; | \; m_i \ne 0 \}$.
	Then
	\[
		\nu_2 \left (\frac{(m_1+m_2+\dots+m_{l}-1)!}{m_1! m_2! \cdots m_{l}!} \right ) \geqslant 
			-\nu_2(m).
	\]
\end{lemma}
\begin{proof}
	Without loss of generality, assume that $m= m_1$ and let $s = \sum_{i=2}^{l}m_i$.
	Observe that
	\[
		\frac{(m_1+m_2+\dots+m_{l}-1)!}{m_1! m_2! \dots m_{l}!} = \frac{(s+m_1-1)\dots(s+1)}{m_1!}\frac{s!}{m_2! \cdots m_{l}!}
	\]
	Clearly the right hand side is the product of $1/m$ and multinomial coefficients.
	% By a classical result of Kummer~\cite{Kummer52}, we have that
% 	\[
% 		\nu_2 \left (\frac{n!}{m_1! m_2! \cdots m_{2k-1}!} \right ) = \sum_{i=1}^{2k-1}B_2(m_i) - B_2(n).
% 	\]
% 	The result follows since $\frac{(m_1+m_2+\dots+m_{2k-1}-1)!}{m_1! m_2! \cdots m_{2k-1}!} = \frac{1}{n}\frac{n!}{m_1! m_2! \cdots m_{2k-1}!}$.
\end{proof}
% \begin{proof}
% 	Observe that $\frac{(m_1+m_2+\dots+m_{2k-1}-1)!}{m_1! m_2! \cdots m_{2k-1}!}$ is equal to
% 	\[
% 		 \frac{m_j!(m_1+m_2+\dots+m_{2k-1}-m_j+m-1)!}{m_1! m_2! \cdots m_{2k-1}!}\frac{1}{m!}
% 	\]
% 	Clearly, $\frac{m_j!(m_1+m_2+\dots+m_{2k-1}-m_j+m-1)!}{m_1! m_2! \cdots m_{2k-1}!}$ is an integer, since it is an integer multiple of a multinomial coefficient.
% 	By Legendre's formula, we have $\nu_2(m!) \leqslant m-1$, as required.
% \end{proof}

Next we show a congruence between coefficients of the characteristic polynomial of $J-2A$ where $A$ is the adjacency matrix of an Euler graph.

\begin{lemma}\label{lem:adet}
	Let $\Gamma$ be an Euler graph of order $n$ odd, let $A$ be its adjacency matrix, and suppose that $\chi_{J-2A}(x) = \sum_{i=0}^n a_i x^{n-i}$.
	% Let $k\geqslant 2$ be an integer.
	Then, for $k \in \{2,\dots,(n-1)/2\}$, we have 
	\begin{align*}
		a_{2k+1} \equiv \sum_{\substack{d \;|\; 2k}} \;\; \sum_{\substack{m_1 + 2 m_2 + \dots + d m_{d} = d \\ m_1 \geqslant 0, \dots, m_{d}\geqslant 0 \\ m_{2k} = 0 } } C_d(m_1,\dots,m_{d}) P_d(m_1,\dots,m_{d})  \mod{2^{2k+1}},
	\end{align*}
	where
	\[
		C_d(m_1,\dots,m_{d}) :=  2^{2k} \frac{d\varphi(2k/d)}{2k} \frac{(m_1+m_2+\dots+m_{d}-1)!}{m_1! m_2! \cdots m_{d}!}
	\]
	and
	\[
		P_d(m_1,\dots,m_{d}) := \prod_{j=1}^{d/2}  \left (\frac{a_{2j+1}}{2^{2j} n} \right )^{m_{2j}}\prod_{j=1}^{d/2+1}  \left (\frac{a_{2j+1}+a_{2j}+a_{2j-1}a_3}{2^{2j-1}} \right )^{m_{2j-1}}.
	\]
\end{lemma}
\begin{proof}
	First write $\chi_{A}(x) = \sum_{i=0}^n b_i x^{n-i}$.
		By Newton's identities, we have
		\begin{align}
				\tr( A^{i} ) &= \sum_{\substack{m_1 + 2 m_2 + \dots + i m_i = i \\ m_1 \geqslant 0, \dots, m_i \geqslant 0 } } \frac{i(m_1+m_2+\dots+m_i-1)!}{m_1! m_2! \cdots m_i!} \prod_{j=1}^i  (-b_j)^{m_j}. \label{eqn:New1} %\\			
			 % \text{ and } \nonumber \\
% 				-i b_i &= \sum_{j=1}^i  \tr( A^{j} ) b_{i-j}. \label{eqn:New2}
		\end{align}
		Using Lemma~\ref{lem:coefficientMap} together with Lemma~\ref{lem:wp4}, we have $a_{2k+1} \equiv 2^{2k}b_{2k} \mod{2^{2k+1}}$.
		% Furthermore we can .
		Now, writing $b_{2k} = b_{2k} + \tr(A^{2k})/2k - \tr(A^{2k})/2k$, apply Lemma~\ref{lem:trrelEuler} to obtain the congruence
		\begin{align}
			a_{2k+1} &\equiv 2^{2k} \left (  b_{2k} + \frac{\tr(A^{2k})}{2k} +\sum_{\substack{d \;|\; 2k \\ d \ne 2k}} \varphi(2k/d)  \frac{\tr(A^d)}{2k} \right ) \mod{2^{2k+1}}. \label{eqn:maincong}
		\end{align}
		
		Observe that, using \eqref{eqn:New1}, we can write
		\[
			  2^{2k}\left (  b_{2k}+\frac{\tr(A^{2k})}{2k} \right ) =  \sum_{\substack{m_1 + \dots + 2k m_{2k} = 2k \\ m_1 \geqslant 0, \dots, m_{2k-1}\geqslant 0 \\ m_{2k} = 0 } } C_{2k}(m_1,\dots,m_{2k-1},m_{2k}) \prod_{j=1}^{2k-1}  (-b_j)^{m_j}.
		\]
		% where the coefficient, $c(m_1,\dots,m_{2k-1})$, is expressed as
% 		% \[
% % 			c(m_1,\dots,m_{2k-1}) = \frac{(m_1+m_2+\dots+m_{2k-1}-2)!}{m_1! m_2! \cdots m_{2k-1}!} \sum_{i=1}^{2k-1}m_i(2k-i).
% % 		\]
% 		\[
% 			c(m_1,\dots,m_{2k-1}) = 2^{2k} \frac{(m_1+m_2+\dots+m_{2k-1}-1)!}{m_1! m_2! \cdots m_{2k-1}!}.
% 		\]
		
		For fixed integers $m_1,\dots,m_{2k-1}$, define $\nu = \min(\{\nu_2(m_i) \; | \; m_i \ne 0 \})$ and let $m \in \{m_i \; | \; \nu_2(m_i) = \nu \}$.
		By Lemma~\ref{lem:multinomial2val}, 
		%the $2$-adic valuation of $C_{2k}(m_1,\dots,m_{2k-1},0)$ has lower bound given by the following
		we have 
		$$\nu_2(C_{2k}(m_1,\dots,m_{2k-1},0)) \geqslant 2k-\nu_2(m).$$
		
		Observe that, for odd $j$, by Corollary~\ref{cor:evenbr}, we have $b_j \equiv -b_j \mod 4$.
		Now, using Corollary~\ref{cor:btoa} together with Lemma~\ref{lem:powersmod}, we can write 
		\[
			\prod_{j=1}^{2k-1}  (-b_j)^{m_j} \equiv P_{2k}(m_1,\dots,m_{2k-1},0) \mod{2^{\nu_2(m)+2}}.
		\]
		Since the $2$-adic valuation of $C_{2k}(m_1,\dots,m_{2k-1},0)$ is at least $2k-\nu_2(m)$, 
		it follows that $2^{2k}\left (  b_{2k} + \frac{\tr(A^{2k})}{2k} \right )$ is congruent modulo $2^{2k+2}$ to 
		\[
			\sum_{\substack{m_1 + 2 m_2 + \dots + 2k m_{2k} = 2k \\ m_1 \geqslant 0, \dots, m_{2k-1}\geqslant 0 \\ m_{2k} = 0 } } C_{2k}(m_1,\dots,m_{2k}) P_{2k}(m_1,\dots,m_{2k}).
		\]
		
		Next, for $d$ a proper divisor of $2k$, using \eqref{eqn:New1}, we can write
		\[
			 2^{2k} \varphi(2k/d) \frac{\tr(A^d)}{2k}  =  \sum_{\substack{m_1 + 2 m_2 + \dots + d m_{d} = d \\ m_1 \geqslant 0, \dots, m_{d}\geqslant 0 } } C_d(m_1,\dots,m_{d}) \prod_{j=1}^{d}  (-b_j)^{m_j}.
		\]
		% where the coefficient, $C_d(m_1,\dots,m_{d})$, is expressed as
% 		% \[
% % 			c(m_1,\dots,m_{2k-1}) = \frac{(m_1+m_2+\dots+m_{2k-1}-2)!}{m_1! m_2! \cdots m_{2k-1}!} \sum_{i=1}^{2k-1}m_i(2k-i).
% % 		\]
% 		\[
% 			C_d(m_1,\dots,m_{d}) =  2^{2k} \frac{d\varphi(2k/d)}{2k}  \frac{(m_1+m_2+\dots+m_{d}-1)!}{m_1! m_2! \cdots m_{d}!}.
% 		\]
		Again, for fixed integers $m_1,\dots,m_{2k}$ define $\nu = \min(\{\nu_2(m_i) \; | \; m_i \ne 0 \})$ and let $m \in \{m_i \; | \; \nu_2(m_i) = \nu \}$.	
		This time, we have the lower bound, 
		
		$$\nu_2(C_d(m_1,\dots,m_{d})) \geqslant 2k-\nu_2(m)-1.$$
		Indeed, using Euler's formula, we see that 
		\[
			\varphi(2k/d)\frac{d}{2k} = \prod_{\substack{p \, | \, 2k/d \\ p \text{ prime}}} \frac{p-1}{p}
		\]
		and hence $\nu_2(d\varphi(2k/d)/2k) \geqslant -1$.
		Then, by Lemma~\ref{lem:multinomial2val}, the $2$-valuation of $C_{d}(m_1,\dots,m_{d})$ has lower bound  $\nu_2(C_d(m_1,\dots,m_{d})) \geqslant 2k-\nu_2(m)-1$.
		% Using Corollary~\ref{cor:btoa} again in a way similar to before, we deduce that
		Using Corollary~\ref{cor:btoa} together with Lemma~\ref{lem:powersmod}, we can write 
		\[
			\prod_{j=1}^{d}  (-b_j)^{m_j} \equiv P_{d}(m_1,\dots,m_{d}) \mod{2^{\nu_2(m)+2}}.
		\]
		Hence $2^{2k} \varphi(2k/d) \tr(A^d)/2k$ is congruent modulo $2^{2k+1}$ to
				\[
					 \sum_{\substack{m_1 + 2 m_2 + \dots + d m_{d} = d \\ m_1 \geqslant 0, \dots, m_{d}\geqslant 0 } } C_{d}(m_1,\dots,m_{d}) P_{d}(m_1,\dots,m_{d}).
				\]
		The lemma then follows from \eqref{eqn:maincong}.
		\end{proof}
		
		Now we can bound the number of congruence classes modulo $2^e \mathbb Z[x]$ of characteristic polynomials of Seidel matrices of odd order.
		
		\begin{corollary}\label{cor:countCharPolySeidelOdd}
			Let $n$ be an odd integer and $e$ be a positive integer.
			Then the cardinality of $\mathcal P_{n,e}$ is at most $2^{\binom{e-2}{2}+1}$.
		\end{corollary}
		\begin{proof}
			By Remark~\ref{rem:a1a2}, the sets $\mathcal P_{n,1}$ and $\mathcal P_{n,2}$ both have cardinality $1$.
			Moreover, by Lemma~\ref{lem:coefficientMap} and Lemma~\ref{lem:oddcond1}, $\mathcal P_{n,3}$ has cardinality at most $2$.
			Assume that $\mathcal P_{n,e}$ has cardinality at most $2^{\binom{e-2}{2}+1}$.
			
			It suffices to show that each polynomial in $\mathcal P_{n,e}$ can be lifted to at most $2^{e-2}$ polynomials in $\mathcal P_{n,e+1}$.
			Let $S$ be a Seidel matrix of order $n$.
			By Theorem~\ref{thm:OddNEuler}, there is a unique Euler graph $\Gamma$ in the switching class of $S$.
			Let $A$ be the adjacency matrix of $\Gamma$.
			Then $\chi_{J-I-2A}(x) \mod 2^e \mathbb Z[x]$ is an element of $\mathcal P_{n,e}$.
			Write $\chi_{J-2A}(x) = \sum_{i=0}^n a_i x^{n-i}$. %where each $a_i$ is reduced modulo $2^e$.
			% Then by Lemma~\ref{lem:adet} and Lemma~\ref{lem:oddcond1}, for all $i \geqslant e+1$ we have that $2^e$ divides $a_i$.
			Then, by Lemma~\ref{lem:coefficientMap}, for all $i \geqslant e+2$ we have $a_i \equiv 0 \mod 2^{e+1}$.
			By Remark~\ref{rem:a1a2}, we have $a_0 = 1$, $a_1= -n$, and $a_2 = 0$.
			
			Assume that the congruence class of each of $a_3, a_4, \dots, a_{e}$ modulo $2^{e}$ is given.
			Then, for each of $a_3, a_4, \dots, a_{e}$, there are two possibilities for its congruence class modulo $2^{e+1}$.
			Furthermore, if $e+1$ is even then, by Lemma~\ref{lem:oddcond1}, the coefficient $a_{e+1}$ is divisible by $2^{e+1}$.
			On the other hand, if $e+1$ is odd, then by Lemma~\ref{lem:adet}, the congruence class of $a_{e+1}$ is determined by that of $a_3, a_4, \dots, a_{e}$.
			Therefore, given that the congruence class of $\chi_{J-2A}(x)$ is fixed modulo $2^e \mathbb Z[x]$, we have that there are $2^{e-2}$ possibilities for congruence class of $\chi_{J-2A}(x)$ modulo $2^{e+1} \mathbb Z[x]$.
			% paragraph claim (end)
		\end{proof}
		
		Just as for Corollary~\ref{cor:countCharPolySeidelEven}, for small values of $e$ ($e \leqslant 7$), we have checked that there exist odd $n$ giving equality in the bound in Corollary~\ref{cor:countCharPolySeidelOdd}.
		The sharpness of this bound for $n=49$ and $e=5$ is a crucial ingredient in the proofs of Theorem~\ref{thm:first_pol} and Theorem~\ref{thm:second_pol}.
		Furthermore, we conjecture that, for all integers $e \geqslant 3$, there exists $N \in \mathbb N$ such that $|\mathcal P_{n,e}| = 2^{\binom{e-2}{2}+1}$ for all odd $n \geqslant N$.
		
		% ssection a_modular_characterisation_of_the_characteristic_polynomial_of_a_seidel_matrix (end)
		
		\section{On 50 equiangular lines in $\mathbb R^{17}$} % (fold)
		\label{sec:50lines}
		
		\subsection{From equiangular lines to Seidel matrices} % (fold)
		\label{sub:from_equiangular_lies_to_seidel_matrices}
		
		% subsection from_equiangular_lies_to_seidel_matrices (end)
		
		Now we apply our restrictions on the characteristic polynomial of a Seidel matrix to equiangular lines.
		Suppose we have a system $\mathcal L$ of $50$ equiangular lines in $\mathbb R^{17}$.
		Let $\{ \mathbf{v}_1,\dots, \mathbf{v}_{50} \}$ be a set of unit spanning vectors for the lines in $\mathcal L$.
		The inner product of any two distinct vectors $\mathbf{v}_i^\top \mathbf{v}_j = \pm \alpha$ for some $\alpha$ in the open interval $(0,1)$.
		The Gram matrix $G$ for this set of vectors has diagonal entries equal to $1$ and off-diagonal entries equal to $\pm \alpha$.
		The matrix $S = (G-I)/\alpha$ is the Seidel matrix corresponding to the system of lines $\mathcal L$.
		Note that the smallest eigenvalue of $S$ is $-1/\alpha$ with multiplicity $50-17=33$.
		By \cite[Theorem 3.4.2]{lemmens73} together with \cite[Lemma 6.1]{vLintSeidel66}, we must have $1/\alpha = 5$.
		Thus the smallest eigenvalue of $S$ is $-5$ with multiplicity $33$.
		% Let $-5<\lambda_1\leqslant \dots \leqslant \lambda_{17}$ be the $17$ other eigenvalues of $S$.
		
		% We refer to the multiset of eigenvalues of a matrix $M$ as the spectrum of $M$ and we write $\{ [a_1]^{m_1}, \dots, [a_k]^{m_k} \}$ to denote a multiset with the element $a_i$ having multiplicity $m_i$ for each $i \in \{1,\dots,k\}$.
		The purpose of this section is to prove the following proposition.
		
		\begin{proposition}\label{pro:onlyints}
			Suppose $S$ is a Seidel matrix of order $50$ having smallest eigenvalue $-5$ with multiplicity $33$.
			Then the characteristic polynomial $\chi_S(x)$ of $S$ must be one of the three polynomials
			\begin{align*}
				&(x + 5)^{33} (x - 9)^{10} (x - 11)^{5} (x^2 - 20x + 95), \\ 
				&(x+5)^{33}(x-9)^{12}(x-11)^4(x-13),  \\ 
				& \text{ or }(x+5)^{33}(x-7)(x-9)^9(x-11)^7.
			\end{align*}
			 % $\{[-5]^{33},[9]^{12},[11]^4,[13]^1\}$ or the spectrum $\{[-5]^{33},[7]^1,[9]^{9},[11]^7\}$.
		\end{proposition}

		Let $\lambda_1,\dots,\lambda_{17}$ be the $17$ other eigenvalues of $S$, which satisfy the inequalities  $-5<\lambda_1\leqslant \dots \leqslant \lambda_{17}$ .
		Observe that $S$ has $\operatorname{tr} S = 0$ and $\operatorname{tr} S^2 = 50\cdot 49$.
		Since the smallest eigenvalues of $S$ is $-5$ with multiplicity $33$, we have
		% Using the above trace equalities for $S$ we see that

						\begin{align}
							\sum_{i=1}^{17} \lambda_i &= 33 \cdot 5 = 165 \label{eqn:tr} \\
							\sum_{i=1}^{17} \lambda_i^2 &= 50\cdot 49 - 33 \cdot 5^2 = 1625. \label{eqn:tr2}
						\end{align}
		\begin{remark}
			\label{rem:top3coeffs}
			Write $\chi_S(x) = \sum_{i=0}^{50} c_i x^{n-i}$.
			Since $\operatorname{tr} S = 0$ and $\operatorname{tr} S^2 = n(n-1)$, using Newton's identities, we have that $c_0 = 1$, $c_1 = 0$, and $c_2 = -50\cdot 49/2$.
		\end{remark}
		
		\subsection{Totally positive algebraic integers of small trace} % (fold)
		\label{sub:totally_positive_algebraic_integers_of_small_trace}

		A polynomial $p(x) \in \mathbb Z[x]$ (resp.\ algebraic integer) is called \textbf{totally positive} if all of its zeros (resp.\ conjugates) are positive.
		The \textbf{trace} of a polynomial $p(x) \in \mathbb Z[x]$ (resp.\ algebraic integer) is defined to be the sum of its zeros (resp.\ conjugates); the trace of $p(x)$ is denoted by $\operatorname{tr} p(x)$.
		In this section, we turn the problem of determining the unknown eigenvalues $\lambda_i$ (for $i \in \{1,\dots,17\}$) of $S$ into a problem about totally positive algebraic integers of small trace.
		
				Putting \eqref{eqn:tr} and \eqref{eqn:tr2} together, we have
				\[
					\sum_{i=1}^{17} (\lambda_i-10)^2 = 25.
				\]

				Define the polynomials $F(x)$ and $G(x)$ by $F(x) := \chi_S(x)/(x+5)^{33} = \prod_{i=1}^{17}(x-\lambda_i)$ and $G(x) := \prod_{i=1}^{17}(x-(\lambda_i-10)^2)$.
				By Corollary~\ref{cor:evenmod2}, none of the $\lambda_i$ can be an even (rational) integer.
				Hence $G(x)$ is a totally positive, monic polynomial in $\mathbb{Z}[x]$ with trace $25$.
				Also note that the difference between the trace and the degree of $G(x)$ is $8$.
				Therefore, each irreducible factor of $G(x)$ must also have its trace minus its degree at most $8$.
				
				Denote by $T(d,t)$ the set of irreducible, totally positive, monic, degree-$d$, integer polynomials of trace $t$.
				Each irreducible factor of $G(x)$ must belong to $T(d,t)$ for some $d$ and $t$ where $1 \leqslant d \leqslant 17$ and $d \leqslant t \leqslant d+8$.
		
				% Let $p(x) = (x-\theta_1)\dots(x-\theta_l)$ be an irreducible factor of $f(x)$ and let $q(x) = (x-(\theta_1-10)^2)\dots(x-(\theta_l-10)^2)$.
		% 		Now we consider two situations: either $q(x)$ is irreducible or $q(x) = r(x)^2$.
		% 		If $q(x)$ is irreducible then $q((x-10)^2) = (-1)^l p(x)p(20-x)$, otherwise if $q(x) = r(x)^2$ then $p(x) = r((x-10)^2)$.
				%
				% \begin{lemma}\label{lem:qmod2}
				% 	$g(x) \equiv (x+1)^{17} \mod 2 \mathbb Z[x]$.
				% \end{lemma}
				% \begin{proof}
				% 	By Lemma~\ref{lem:mod2}, the polynomial $f(x) \equiv (x+1)^{17}$.
				% 	We can write $g(x^2) = -f(x+10)f(-x+10) \equiv (x+1)^{34} \mod {2\mathbb Z[x]}$.
				% 	Since $(x+1)^{34} \equiv \sum_{i=0}^{17} \binom{34}{2i}x^{2i} \mod {2\mathbb Z[x]}$, we see that $g(x) \equiv (x+1)^{17} \mod 2 \mathbb Z[x]$, as required.
				% \end{proof}
		
				\begin{lemma}\label{lem:qmod2}
					Suppose that $f(x) = \prod_{i=1}^{d}(x-\lambda_i) \equiv (x+1)^{d} \mod {2 \mathbb Z[x]}$.
					Let $\mu$ be an integer and $g(x) = \prod_{i=1}^{d}(x-(\lambda_i-2\mu)^2)$.
					Then $g(x) \equiv (x+1)^{d} \mod {2 \mathbb Z[x]}$.
				\end{lemma}
				\begin{proof}
					We can write $g(x^2) = (-1)^d f(x+2\mu)f(-x+2\mu) \equiv (x+1)^{2d} \mod {2\mathbb Z[x]}$.
					Since $(x+1)^{2d} \equiv \sum_{i=0}^{d} \binom{2d}{2i}x^{2i} \mod {2\mathbb Z[x]}$, we see that $g(x) \equiv (x+1)^{d} \mod 2 \mathbb Z[x]$, as required.
				\end{proof}
				% \begin{proof}
				% 	Let $q(x)$ be an irreducible factor of $g(x)$.
				% 	Then we have two different situations.
				% 	First assume that $q((x-10)^2) = (-1)^l p(x)p(20-x)$.
				% 	Then
				% 	\[
				% 		q(x^2) \equiv p(x)^2 \equiv (x+1)^{2l} \equiv (x^2+1)^l \mod 2 \mathbb Z[x].
				% 	\]
				% 	Hence $q(x) \equiv (x+1)^l \mod 2 \mathbb Z[x]$.
				% 	Otherwise we have $q(x) = r(x)^2$ and $p(x) = r((x-10)^2)$.
				% 	This implies that $l=2m$ for some integer $m$.
				% 	Then
				% 	\[
				% 		p(x) \equiv (x+1)^l \equiv r(x^2) \equiv (x^2+1)^m \mod 2 \mathbb Z[x].
				% 	\]
				% 	Hence, again, we have $q(x) \equiv (x+1)^l \mod 2 \mathbb Z[x]$.
				% \end{proof}
		
		By Corollary~\ref{cor:evenmod2}, the polynomial $F(x)$ is congruent to $(x+1)^{d}$ modulo ${2 \mathbb Z[x]}$.
			Hence, by Lemma~\ref{lem:qmod2}, the polynomial $G(x)$ is also congruent to $(x+1)^d$ modulo $2\mathbb Z[x]$, and since $\mathbb Z[x]/2\mathbb Z[x]$ is a UFD, each irreducible factor of $G(x)$ of degree $e$ is congruent to $(x+1)^e$ modulo $2\mathbb Z[x]$.
				Next we consider all possible candidates for the irreducible factors of $G(x)$.
		
		Define the set $\mathcal T(d,t) := \{ p(x) \in T(d,t) \; | \; p(x) \equiv (x+1)^d \mod{2\mathbb Z[x]}  \}$.
		We distinguish between two kinds of irreducible factors $\eta(x)$ of $G(x)$ depending on whether or not $\eta(x^2)$ is irreducible.
		Define the sets
		\begin{align*}
			\mathcal T_{\text{irr}}(d,k) &:= \{ p(x) \in \mathcal T(d,k+d) \; | \; p(x^2) \text{ is irreducible}  \} \\
			\mathcal T_{\text{red}}(d,k) &:= \{ p(x) \in \mathcal T(d,k+d) \; | \; p(x^2) \text{ is reducible}  \}.
		\end{align*}
		
		\begin{lemma}\label{lem:irrsquare}
			Suppose that $g(x) \in \mathcal T_{\text{irr}}(d,k)$ is a factor of $G(x)$.
			Then $g(x)^2$ divides $G(x)$.
		\end{lemma}
		\begin{proof}
			We have $G(x^2) = (-1)^{17} F(10+x)F(10-x)$.
			Therefore $(-1)^{17} F(10+x)F(10-x) = g(x^2)h(x^2)$, for some monic polynomial $h(x) \in \mathbb Z[x]$.
			Since $g(x^2)$ is irreducible, it must divide $F(10+x)$ or $F(10-x)$. %it follows that $g(x^2)$ divides both $F(10+x)$ and $F(10-x)$.
			Furthermore, since $g(x^2)$ is invariant under substituting $-x$ for $x$, it must divide both $F(10+x)$ and $F(10-x)$.
			Hence $g(x^2)$ is also a factor of $h(x^2)$, as required.
		\end{proof}
		
		We will use the following consequence of Lemma~\ref{lem:irrsquare}.
		
		\begin{corollary}\label{cor:trBndFactors1}
			Let $g(x)$ be an irreducible factor of $G(x)$ of degree $d$.
			Then either $g(x) \in \mathcal T_{\text{irr}}(d,k)$ with $k \leqslant 4$ or $g(x) \in \mathcal T_{\text{red}}(d,k)$ with $k \leqslant 8$.
		\end{corollary}

		The next theorem is a result of McKee~\cite{McKee10}.

		\begin{theorem}\label{thm:trBnd}
			Let $f(x) \in \mathbb Z[x]$ be a monic totally-positive irreducible polynomial of degree $d \geqslant 5$.
			Then $\operatorname{tr} f(x) \geqslant \lceil 1.78839 d \rceil$.
		\end{theorem}
		
		It follows that we can further improve the bounds on the degree of the putative irreducible factors of $G(x)$.

		\begin{corollary}\label{cor:trBndFactors2}
			Let $g(x)$ be an irreducible factor of $G(x)$ of degree $d$.
			Suppose that $g(x) \in \mathcal T_{\text{irr}}(d,k)$ (resp. $g(x) \in \mathcal T_{\text{red}}(d,k)$) for some $d$ and $k$.
			Then $d \leqslant 5$ (resp. $d \leqslant 10$).
		\end{corollary}
		
		By Corollary~\ref{cor:trBndFactors1} and Corollary~\ref{cor:trBndFactors2}, each irreducible factor of $G(x)$ belongs to either $\mathcal T_{\text{irr}}(d,k)$ with $d \leqslant 5$ and $k \leqslant 4$ or $\mathcal T_{\text{red}}(d,k)$ with $d \leqslant 10$ and $k \leqslant 8$.
		
		% subsection totally_positive_algebraic_integers_of_small_trace (end)
		
		\subsection{Polynomial enumeration algorithm} % (fold)
		\label{sub:polynomial_enumeration_algorithm}

		We use a method due to Robinson to find all possible irreducible factors of $G(x)$.
		This method has been detailed by Smyth~\cite{Smyth84} and McKee and Smyth~\cite{MS04}.
		For completeness, following McKee and Smyth~\cite[Section 3.2]{MS04}, we describe the algorithm below.
		We will also use this algorithm in Proposition~\ref{prop:kill_sub}.
		% *** DESCRIBE ALGORITHM ***
		
		First we state a result about the interlacing of the zeros of a totally real polynomial and its derivative.
		This result is a straightforward consequence of Rolle's theorem.
		
		\begin{proposition}\label{pro:derivative}
			Let $d \geqslant 2$ and let $p(x)$ be a monic degree-$d$ polynomial having zeros $\alpha_1 < \dots < \alpha_{d}$.
			Denote by $\beta_1 < \dots < \beta_{d-1}$ the zeros of its derivative $p^\prime(x)$.
			Then
			\[
				\alpha_1 < \beta_1 < \alpha_2 < \dots < \alpha_{d-1} < \beta_{d-1} < \alpha_d.
			\]
		\end{proposition}
		
		% The main result behind Robinson's algorithm is the following (see \cite{Smyth84}).
%
% 		\begin{proposition}\label{pro:coeffBnd}
% 			Let $d \geqslant 2$ and let $p(x)$ be a monic polynomial given by
% 			\[
% 				p(x) = \int_{0}^x p^\prime(t) \mathrm dt,
% 			\]
% 			where $p^\prime(x)$ is a polynomial of degree $d-1$ having zeros $\beta_1 < \dots < \beta_{d-1}$.
% 			Then $p(x) - c$ is totally real if and only if $(-1)^dc < 0$ and
% 			\[
% 				\max_{i=1}^{\lceil k/2 \rceil} p(\beta_{2i-1}) \leqslant c \leqslant \min_{i=1}^{\lfloor k/2 \rfloor} p(\beta_{2i}).
% 			\]
% 		\end{proposition}
		Fix an integer $t$ and real numbers $l$ and $u$.
		Let $\mathfrak P \subset \mathbb Z[x]$ be the set of monic
		 % irreducible 
		 polynomials
		 % congruent to $(x+1)^d \mod 2\mathbb Z[x]$ 
		having trace $t$ and all zeros inside the interval $[l,u]$.
		We would like to find all polynomials in $\mathfrak P$.
		% Let $p(x)$ be a monic degree-$d$ polynomial in $\mathbb Z[x]$ having zeros $\alpha_1 < \dots < \alpha_d$ such that $\sum_{i=1}^d \alpha_i = d+k$.
% 		Suppose that we have $l \leqslant \alpha_1$ and $\alpha_d \leqslant u$.
		Suppose $p(x) \in \mathfrak P$.
		Then
		\[
			p(x) = x^d - t x^{d-1} + a_2 x^{d-2} + \dots + a_{d-1}x +a_{d},
		\]
		for some integers $a_2, \dots, a_d$.
		For $r = d, d-1, \dots, 1$, define
		\[
			p_r(x) = \frac{r!}{d!}\frac{\mathrm d^{d-r}}{\mathrm d x^{d-r}} p(x) = x^r - t\frac{r}{d}x^{r-1} + \dots +  a_r\frac{r!(d-r)!}{d!}.
		\]
		Let $\gamma^{(r)}_1 \leqslant \gamma^{(r)}_2 \leqslant \dots \leqslant \gamma^{(r)}_r$ be the zeros of $p_r(x)$.
		% Then for all $1\leqslant i \leqslant r$, by Proposition~\ref{pro:derivative}, we have that $\alpha_i < \gamma_i^{(r)} < \alpha_{d-r+i}$.
		Then, by Proposition~\ref{pro:derivative}, we have that $l < \gamma_1^{(r)}$ and $\gamma_r^{(r)} < u$.
		% Hence, in particular $l \leqslant \gamma_1^{(r)}$ and $\gamma_r^{(r)} \leqslant u$.
		% We can successively apply these bounds, together with Proposition~\ref{pro:coeffBnd}, to $p_2(x), p_3(x), \dots, p_{d}(x) = p(x)$ to find bounds for $a_2,\dots,a_d$.
		Given a candidate for $p_i(x)$, having $i$ distinct zeros, that is, given values for $a_2,\dots,a_i$, we seek the (possibly empty) range of values for $a_{i+1}$ such that $p_{i+1}(x)$ has $i+1$ distinct zeros in the interval $[l,u]$.
		To find this range of values for $a_{i+1}$ we need to find the middle (with a left bias for discriminant of odd degree) two zeros of the discriminant of $p_{i+1}(x)$, which is a polynomial in $a_{i+1}$.
		We refer to \cite[Section 3.2]{MS04} for the details.
		% Once we have the interval for $a_{i+1}$,
% 		% we can use the congruence $p(x) \equiv (x+1)^d \mod 2 \mathbb Z[x]$, and restrict to $a_{i+1}$ satisfying $a_{i+1} \equiv \binom{d}{i+1} \mod 2$.
% 		we can set $a_{i+1}$ to take each of the integer values in this interval.
% 		Then repeat the above process for each value of $a_{i+1}$ to find $a_{i+2}$.
		
		This gives us a tree with root (or first generation) $a_1 = -(d+k)$.
		The $r$th generation consisting of nodes $v$ such that each path $a_1 = v_1, v_2, \dots, v_r = v$ from the root $a_1$ to $v$ corresponds to the polynomial $p_r$, where $a_i = v_i$ for all $1\leqslant i \leqslant r$. 
		Each path of length $d$ in the tree corresponds to a candidate for $p(x)$.
		% Finally we check that $p(x)$ is irreducible.
		This way we obtain all elements of $\mathfrak P$.
		
		Now we remark on the applications of this algorithm.
		
		When finding polynomials in $\mathcal T(d,t)$, we can speed up the algorithm by taking advantage of the fact that each polynomial in $\mathcal T(d,t)$ is congruent to $(x+1)^d$ modulo $2\mathbb Z[x]$.
		That is, for each $i \in \{2,\dots,d\}$, we restrict each $a_{i}$ so that $a_{i} \equiv \binom{d}{i} \mod 2$.
	
		When applying the algorithm to the proofs of Theorems~\ref{thm:first_pol},~\ref{thm:second_pol},~\ref{thm:third_pol}, and Proposition~\ref{prop:kill_sub}, we not only fix the trace, we also fix the coefficient $a_2$.
		% This idea can also be used in the proofs of Theorem~\ref{thm:first_pol} and Theorem~\ref{thm:second_pol}.
		
		% subsection polynomial_enumeration_algorithm (end)
		
		\subsection{Candidates for irreducible factors of $G(x)$} % (fold)
		\label{sub:candidates_for_irreducible_factors_of_g_x}
		
		Now we detail the results of the computation of all polynomials in $\mathcal T_{\text{irr}}(d,k)$ for $d \in \{1,\dots,5\}$ and $k \in \{0,2,4\}$ and all polynomials in $\mathcal T_{\text{red}}(d,k)$ for $d \in \{1,\dots,10\}$ and $k \in \{0,2,4,6,8\}$.
		We list all the polynomials in these sets in Tables~\ref{tab:IRR}~and~\ref{tab:RED}.
		
		\begin{table}[htbp]
			\begin{center}
			\begin{tabular}{c|c}
				$\mathcal T_{\text{irr}}(1,2)$ & $x-3$ \\
				 \hline
 				$\mathcal T_{\text{irr}}(1,4)$ & $x-5$ \\
 				 \hline
  				$\mathcal T_{\text{irr}}(2,2)$ & $x^2-4x+1$ \\
  				 \hline
   				$\mathcal T_{\text{irr}}(2,4)$ & $x^2-6x+3$ \\
				& $x^2-6x+7$ \\
   				 \hline
    			$\mathcal T_{\text{irr}}(3,4)$ & $x^3-7x^2+9x-1$ \\
				 & $x^3-7x^2+11x-3$ \\
				 & $x^3-7x^2+13x-5$ \\
    			 \hline
     			$\mathcal T_{\text{irr}}(4,4)$ & $x^4-8x^3+16x^2-8x+1$ \\
 				 & $x^4-8x^3+18x^2-10x+1$ \\
 				 & $x^4-8x^3+20x^2-16x+1$  
			\end{tabular}
			\end{center}
			\caption{The elements of $\mathcal T_{\text{irr}}(d,k)$ for $d \in \{1,\dots,5\}$ and $k \in \{0,2,4\}$}
			\label{tab:IRR}
		\end{table}
		
		% \begin{table}[htbp]
		% 	\begin{center}
		% 	\begin{tabular}{c|ccccc}
		% 		 & \multicolumn{5}{|c}{$k$} \\
		% 		 \hline
		% 		$d$ & 0 & 2 & 4 & 6 & 8 \\
		% 		\hline
		% 		1 & 1 & 0 & 0 & 0 & 1 \\
		% 		2 & 0 & 0 & 1 & 0 & 0 \\
		% 		3 & 0 & 0 & 1 & 0 & 5 \\
		% 		4 & 0 & 0 & 0 & 0 & 7 \\
		% 		5 & 0 & 0 & 0 & 0 & 11 \\
		% 		6 & 0 & 0 & 0 & 0 & 8 \\
		% 		7 & 0 & 0 & 0 & 0 & 4 \\
		% 	\end{tabular}
		% 	\end{center}
		% 	\caption{The cardinalities of $\mathcal T_{\text{red}}(d,k)$}
		% 	\label{tab:cardRED}
		% \end{table}
		
		\begin{table}[htbp]
			\begin{center}
				% \resizebox{1\columnwidth}{!}{%
			\begin{tabular}{c|c}
				$\mathcal T_{\text{red}}(1,0)$ & $x-1$ \\
				 \hline
 				$\mathcal T_{\text{red}}(1,8)$ & $x-9$ \\
 				 \hline
  				$\mathcal T_{\text{red}}(2,4)$ & $x^2-6x+1$ \\
  				 \hline
   				$\mathcal T_{\text{red}}(3,4)$ & $x^3-7x^2+11x-1$ \\
   				 \hline
    			$\mathcal T_{\text{red}}(3,8)$ & $x^3-11x^2+7x-1$ \\
				 & $x^3-11x^2+23x-1$ \\
				 & $x^3-11x^2+27x-1$ \\
				 & $x^3-11x^2+31x-25$ \\
				 & $x^3-11x^2+31x-9$ \\
    			 \hline
     			$\mathcal T_{\text{red}}(4,8)$ & $x^4-12x^3+26x^2-12x+1$ \\
 				 & $x^4-12x^3+34x^2-20x+1$ \\
 				 & $x^4-12x^3+38x^2-40x+9$ \\
 				 & $x^4-12x^3+38x^2-16x+1$ \\
 				 & $x^4-12x^3+42x^2-44x+1$ \\
				 & $x^4-12x^3+46x^2-64x+25$ \\
				 & $x^4-12x^3+46x^2-56x+1$ \\
     			 \hline
      			$\mathcal T_{\text{red}}(5,8)$ & $x^5-13x^4+42x^3-46x^2+13x-1$ \\
  				 & $x^5-13x^4+46x^3-42x^2+13x-1$ \\
  				 & $x^5-13x^4+50x^3-66x^2+17x-1$ \\
  				 & $x^5-13x^4+50x^3-62x^2+21x-1$ \\
  				 & $x^5-13x^4+54x^3-90x^2+53x-1$ \\
 				 & $x^5-13x^4+54x^3-86x^2+49x-9$ \\
 				 & $x^5-13x^4+54x^3-78x^2+33x-1$ \\
  				 & $x^5-13x^4+54x^3-74x^2+21x-1$ \\
 				 & $x^5-13x^4+58x^3-106x^2+73x-9$ \\
 				 & $x^5-13x^4+58x^3-102x^2+61x-9$ \\
				 & $x^5-13x^4+58x^3-98x^2+41x-1$ \\
      			 \hline
       			$\mathcal T_{\text{red}}(6,8)$ & $x^6-14x^5+59x^4-96x^3+59x^2-14x+1$ \\
   				 & $x^6-14x^5+63x^4-104x^3+63x^2-14x+1$ \\
   				 & $x^6-14x^5+67x^4-136x^3+111x^2-26x+1$ \\
   				 & $x^6-14x^5+67x^4-132x^3+99x^2-26x+1$ \\
   				 & $x^6-14x^5+67x^4-132x^3+103x^2-22x+1$ \\
  				 & $x^6-14x^5+71x^4-160x^3+151x^2-38x+1$ \\
  				 & $x^6-14x^5+71x^4-160x^3+155x^2-50x+1$ \\
   				 & $x^6-14x^5+71x^4-156x^3+135x^2-26x+1$ \\
				 \hline
				 $\mathcal T_{\text{red}}(7,8)$ & $x^7 - 15x^6 + 81x^5 - 203x^4 + 243x^3 - 125x^2 + 23x - 1$ \\
				 & $x^7 - 15x^6 + 81x^5 - 195x^4 + 215x^3 - 101x^2 + 19x - 1$ \\
				 & $x^7 - 15x^6 + 85x^5 - 227x^4 + 287x^3 - 149x^2 + 23x - 1$ \\
				 & $x^7 - 15x^6 + 85x^5 - 227x^4 + 291x^3 - 165x^2 + 35x - 1$
			\end{tabular}
			% }
			\end{center}
			\caption{The elements of $\mathcal T_{\text{red}}(d,k)$ for $d \in \{1,\dots,10\}$ and $k \in \{0,2,4,6,8\}$}
			\label{tab:RED}
		\end{table}
		
		Note that we need only to consider even values of $k$.
		Indeed, by Lemma~\ref{lem:qmod2}, each degree-$d$, irreducible factor $g(x)$ of $G(x)$ must be congruent to $(x+1)^d \mod 2\mathbb Z[x]$.
		Hence the trace of $g(x)$ is congruent to $d$ modulo $2$.
		Furthermore, by Theorem~\ref{thm:trBnd}, for $t < \lceil 1.78839 d \rceil$ we know that $\mathcal T(d,t)$ is empty.
		
		To simplify our task, we take advantage of related computations that have already been performed.
		
		\paragraph{\emph{Trace minus degree at most $6$.}} % (fold)
		\label{par:trace_minus_degree_at_most_6}
		
		Smyth~\cite{Smyth84} has enumerated the elements of the sets $T(d,t)$ for all integers $d \geqslant 1$ and $t \geqslant d$ such that $t-d \leqslant 6$.
		We can use the lists of Smyth to find the sets $\mathcal T_{\text{irr}}(d,k)$ and $\mathcal T_{\text{red}}(d,k)$ for $k \leqslant 6$.
		
		% paragraph trace_minus_degree_at_most_6 (end)
		
		\paragraph{\emph{Degree $10$.}} % (fold)
		\label{par:degree_10}
		
		For degree-$10$ candidates, using the above restrictions, we need only compute the set $\mathcal T_{\text{red}}(10,8)$.
		The set $T(10,18)$ was computed in \cite{MS04}.
		It consists of three polynomials, from which it is straightforward to check that the set $\mathcal T(10,18)$ is empty.
		
		% paragraph degree_10 (end)
		
		\paragraph{\emph{Degree $9$.}} % (fold)
		\label{par:degree_9}
		
		For degree-$9$ candidates, using the above restrictions, we need only compute the set $\mathcal T_{\text{red}}(9,8)$.
		It was shown in \cite{AP08} that the cardinality of the set $T(9,17)$ is $686$.
		We recomputed this list of $686$ polynomials and found that just two of these polynomials $p_1(x)$ and $p_2(x)$ (say) are congruent to $(x+1)^9 \mod 2\mathbb Z[x]$.
		However, both $p_1(x^2)$ and $p_2(x^2)$ are irreducible.
		Hence, $\mathcal T_{\text{red}}(9,8)$ is empty.
		% We can use this to verify our computation of $T(9,17)$, which we can use to check the correctness of our computation of the set $\mathcal T_{\text{irr}}(9,8)$.
		
		\paragraph{\emph{Trace minus degree $8$.}} % (fold)
		\label{par:degree_87}
		
		It remains for us to compute the sets $\mathcal T_{\text{red}}(d,8)$ for $d \in \{1,\dots,8\}$.
		We used the algorithm described in Section~\ref{sub:polynomial_enumeration_algorithm}.
		The computations ran in SageMath~\cite{sage} on a single core of an Intel Core i7 at 2.9 GHz.
		% The computation of $\mathcal T_{\text{red}}(7,8)$ and $\mathcal T_{\text{red}}(8,8)$ took about 95 and 774 seconds respectively.
		Computing the sets $\mathcal T_{\text{red}}(d,8)$ for $d \in \{1,\dots,8\}$ took 8 hours and 24 minutes in total.
		We have listed the elements of these sets in Table~\ref{tab:RED}.
		We found that, although the set $\mathcal T(8,16)$ has cardinality $48$, the set $\mathcal T_{\text{red}}(8,8)$ is empty.
		
		% paragraph degrees_8_and_7 (end)
		
		% In Tables~\ref{tab:cardIRR}~and~\ref{tab:cardRED} we list the cardinalities  and we list the polynomials in these sets in Tables~\ref{tab:IRR}~and~\ref{tab:RED}.
		
		% In Tables~\ref{tab:IRR}~and~\ref{tab:RED} we list the elements of the nonempty sets $\mathcal T_{\text{irr}}(d,k)$ for $d \in \{1,\dots,5\}$ and $k \in \{0,2,4\}$ and $\mathcal T_{\text{red}}(d,k)$ for $d \in \{1,\dots,10\}$ and $k \in \{0,2,4,6,8\}$.

		% \begin{table}[htbp]
		% 	\begin{center}
		% 	\begin{tabular}{c|ccc}
		% 		 & \multicolumn{3}{|c}{$k$} \\
		% 		 \hline
		% 		$d$ & 0 & 2 & 4  \\
		% 		\hline
		% 		1 & 0 & 1 & 1 \\
		% 		2 & 0 & 1 & 2 \\
		% 		3 & 0 & 0 & 3 \\
		% 		4 & 0 & 0 & 3 \\
		% 	\end{tabular}
		% 	\end{center}
		% 	\caption{The cardinalities of $\mathcal T_{\text{irr}}(d,k)$}
		% 	\label{tab:cardIRR}
		% \end{table}
		
		% subsection candidates_for_irreducible_factors_of_g_x (end)

		\subsection{From $G(x)$ to $F(x)$} % (fold)
		\label{sub:from_g_x_to_f_x}
		
		% Form the set
% 		\[
% 			\mathcal U := \bigcup_{d=1}^{5} \bigcup_{k=1}^{4} \{ g(x)^2 \; | \; g(x) \in  \mathcal T_{\text{irr}}(d,k) \} \cup \bigcup_{d=1}^{10} \bigcup_{k=1}^{8} \mathcal T_{\text{red}}(d,k).
% 		\]
		Now we have computed all possible candidates (see Tables~\ref{tab:IRR}~and~\ref{tab:RED}) for the irreducible factors of $G(x)$, we can find all candidates for the polynomial $G(x)$.
		By Lemma~\ref{lem:irrsquare}, each factor $g(x) \in \mathcal T_{\text{irr}}(d,k)$ of $G(x)$ for some $d$ and $k$ must be a factor with even multiplicity.
		Hence we form the set
		\[
			\mathcal U:= \bigcup_{d=1}^{5} \bigcup_{k=1}^{4} \{ g(x)^2 \;|\; g(x) \in \mathcal T_{\text{irr}}(d,k) \} \cup \bigcup_{d=1}^{10} \bigcup_{k=1}^{8} \mathcal T_{\text{red}}(d,k).
		\]
		Then define $\overline{\mathcal U} := \bigcup_{i=1}^{17} \mathcal U^i$ to be a union of cartesian powers of $\mathcal U$.
		%Write each element $\mathbf{v} \in \overline{\mathcal U}$ as $\mathbf{v} = (v_1(x), v_2(x),\dots,v_{|\mathbf{v}|}(x))$ and define $p_{\mathbf{v}}(x) := \prod_{i=1}^{|\mathbf{v}|} v_i(x)$.
		For a vector of $l$ polynomials, $\mathbf{v} = (v_1(x), v_2(x),\dots,v_{l}(x))$, define $p_{\mathbf{v}}(x) := \prod_{i=1}^{l} v_i(x)$.
		Now form the set
		\[
			\mathcal G = \left \{ p_\mathbf{v}(x)  \; | \; \mathbf{v} \in \overline{\mathcal U},\; \deg p_\mathbf{v}(x) = 17, \text{ and } \operatorname{tr} p_\mathbf{v}(x) = 25 \right  \},
		\]
		which contains all candidates for the polynomial $G(x)$.
		The cardinality of $\mathcal G$ is $55$.
		
		Now we can construct all possible candidates for the polynomial $F(x)$.
		Each polynomial $g(x) \in \mathcal T_{\text{irr}}(d,k)$ corresponds to the polynomial $f(x) = g((x-10)^2)$, which is a potential factor of $F(x)$.
		Each polynomial $g(x) \in \mathcal T_{\text{red}}(d,k)$ corresponds to the polynomial $g((x-10)^2) = f_1(x)f_2(x)$, where both $f_1(x)$ and $f_2(x)$ are potential factors of $F(x)$.
		Therefore, for each $g(x) \in \mathcal T_{\text{irr}}(d,k)$, we define the set $\mathcal H(g(x)) := \{ g((x-10)^2)^{1/2} \}$ and, for each $g(x) \in \mathcal T_{\text{red}}(d,k)$, we define the set $\mathcal H(g(x)) := \{ f(x) \; | \; f(x) \text{ is a factor of } g((x-10)^2) \}$.
		
		For a polynomial $G(x) \in \mathcal G$, denote by $\operatorname{irr}(G(x))$ the multiset of irreducible factors of $G(x)$ and define the set $$P(G(x)) := \prod_{\nu(x) \in \operatorname{irr}(G(x))} \mathcal H(\nu(x)).$$
		% Write each element $\mathbf{v} \in P(G(x))$ as $\mathbf{v} = (v_1(x), v_2(x),\dots,v_{|\mathbf{v}|}(x))$ and define $p_{\mathbf{v}}(x) := \prod_{i=1}^{|\mathbf{v}|} v_i(x)$.
		%and define the set
		% \[
% 			\mathcal H(G(x)) := \left \{ \prod_{\nu(x) \in \operatorname{irr}(G(x))} f_{\nu(x)} \; | \; f_{\nu(x)} \in \mathcal H(\nu(x)) \right \}.
% 		\]
		Now define the set of polynomials
		\[
			\mathcal R(G(x)) := \left \{ p_\mathbf{v}(x) \mid \mathbf{v} \in P(G(x)) \right \}.
		\]
		
		The set $\mathcal F = \bigcup_{G(x) \in \mathcal G} \mathcal R(G(x))$ contains all candidates for the polynomial $F(x)$.
		For each $p(x) \in \mathcal F$, the polynomial $c(x) := (x+5)^{33}p(x)$ is a candidate for the characteristic polynomial of $S$.
		We now check that the polynomial $c(x)$ satisfies the necessary conditions for $c(x)$ being the characteristic polynomial of a Seidel matrix of order $50$.
		By Remark~\ref{rem:top3coeffs}, the top three coefficients of $c(x)$ must be $1$, $0$, and $-50\cdot 49/2 = -1225$ respectively.
		Hence the top three coefficients of $p(x)$ must be $1$, $-165$, and $12800$, respectively.
		Define
		\[
			\mathcal C := \{ (x+5)^{33}p(x) \; | \; p(x) \in \mathcal F \text{ and } p(x) = x^{17} - 165x^{16} + 12800x^{15} + \dots  \}.
		\]
		Recall that $S$ is the Seidel matrix for $\mathcal L$, a putative set of $50$ equiangular lines in $\mathbb R^{17}$.
		The set $\mathcal C$, which consists of $102$ polynomials, contains all candidates for the characteristic polynomial $\chi_{S}(x)$ of $S$.
		% Next we construct all possible $F(x)$.
		%
		% Then, for each $\mathbf{v} \in \mathcal G$, form the set
		% \[
		% 	\mathcal H(\mathbf{v}) = \left \{  \prod_{i=1}^{|\mathbf{v}|} f_i(x) \; | \; f_i(x) \in \mathcal F(v_i(x)) \text{ and } \operatorname{tr} \prod_{i=1}^{|\mathbf{v}|} f_i(x) = -165 \right \}.
		% \]
		% The set $\mathcal P = \bigcup_{\mathbf{v} \in \mathcal G} \mathcal H(\mathbf{v})$ is then a complete set of possibilities for $F(x)$.
		% The set $\mathcal P$ has cardinality $98$.
		%
		% % Multiplying each element in $\mathcal P$ by $(x+5)^{33}$ yields a complete set of possibilities for the characteristic polynomial of $S$.
		% For each $p(x) \in \mathcal P$, the polynomial $q(x) := (x+5)^{33}p(x)$ is a candidate for the characteristic polynomial of $S$.
		Write $c(x-1) = \sum_{i=0}^{50} a_i x^{50-i}$.
		By Corollary~\ref{cor:evencond}, for each $i \in \{0,\dots,50\}$, the coefficient $a_i$ must be divisible by $2^i$.
		This condition weeds out all but three possibilities from $\mathcal C$, and we thus find that $\chi_S(x)$ must be one of
		\begin{align*}
			&(x + 5)^{33} (x - 9)^{10} (x - 11)^{5} (x^2 - 20x + 95), \\ 
			&(x+5)^{33}(x-9)^{12}(x-11)^4(x-13),  \\ 
			& \text{ or }(x+5)^{33}(x-7)(x-9)^9(x-11)^7.
		\end{align*}
		This proves Proposition~\ref{pro:onlyints}.
		% subsection from_g_x_to_f_x (end)

		\section{Nonexistence results for Seidel matrices with certain characteristic polynomials}
		\label{sec:nonexist}
		
		In this section we show that there does not exist a Seidel matrix having any of the characteristic polynomials of Proposition~\ref{pro:onlyints}, thereby showing that there cannot exist $50$ equiangular lines in $\mathbb R^{17}$.
		
		Let $M$ be a real symmetric matrix of order $n$ having $m$ distinct eigenvalues $\lambda_1 > \lambda_2 > \dots > \lambda_m$.
		 % and spectral decomposition $M = \sum_{i=1}^m \lambda_i P_i$.
		We write $\Lambda(M) = \{\lambda_1,\dots,\lambda_m\}$ for the set of distinct eigenvalues of $M$.
		For each $i \in \{1,\dots,m\}$, denote by $\mathcal{E}(\lambda_i)$ the eigenspace of $\lambda_i$ and let $\{\mathbf{e}_1,\dots,\mathbf{e}_n\}$ be the standard basis of $\mathbb R^n$.
		Denote by $P_i$ the orthogonal projection of $\mathbb R^n$ onto $\mathcal{E}(\lambda_i)$.
		We write $\alpha_{ij} = ||P_i \mathbf{e}_j||$ for all $i \in \{1,\dots,m\}$ and $j \in \{1,\dots,n\}$.
		As is customary, we refer to the numbers $\alpha_{ij}$ as the \textbf{angles} of $M$.
		
		\begin{proposition}\label{prop:sum_angles}\cite[Prop. 4.2.1]{crs97}
			Let $M$ be a real symmetric matrix of order $n$ with $\Lambda(M) = \{\lambda_1,\dots,\lambda_m\}$ and angles $\alpha_{ij}$.
			Then
			$\sum^n_{j=1}\alpha^2_{ij}=\dim \mathcal{E}(\lambda_i)$ for each $i\in \{1,\dots,m\}$, and $\sum^m_{i=1}\alpha^2_{ij}=1$ for each $j \in \{1,\dots,n\}$.
		\end{proposition}

		Now we define the \textbf{angle matrix} of $M$ to be the $n$-by-$m$ matrix with $(j,i)$-entry equal to $\alpha^2_{ij}$.
		Note that
		 % (for convenience) 
		the angle matrix defined here is the transpose of the angle matrix defined in \cite[Chapter 4]{crs97}.
		Let $M[r]$ denote the principal submatrix of $M$ obtained by deleting the $r$-th row and column of $M$.
		For a proof of the next result see \cite[(4.2.8)]{crs97} or \cite{GoMK}.
		
		\begin{proposition}\label{prop:submatrix}
			Let $M$ be a real symmetric matrix of order $n$ with $\Lambda(M) = \{\lambda_1,\dots,\lambda_m\}$ and angles $\alpha_{ij}$.
			Then, for each $j \in \{1,\dots,n\}$, we have
		$$\chi_{M[j]}(x)=\chi_M(x)\sum^m_{i=1}\dfrac{\alpha^2_{ij}}{x-\lambda_i}.$$
		\end{proposition}

		% \begin{remark}\label{rmk:angle_vectors}
		% \end{remark}
		By Proposition~\ref{prop:submatrix}, each characteristic polynomial of an $(n-1)$-by-$(n-1)$ principal submatrix of $M$ corresponds to a row of the angle matrix of $M$. 
		We take advantage of this fact in the main results of this section.
		
		In this section we will also make use of Cauchy's interlacing theorem~\cite{Cau:Interlace,Fisk:Interlace05,Hwang:Interlace04}.

		\begin{theorem}\label{thm:cauchyinterlace}
			Let $M$ be a real symmetric matrix having eigenvalues $\lambda_1 \leqslant \lambda_2 \leqslant \dots \leqslant \lambda_n$ and suppose $M[j]$, for some $j \in \{1,\dots,n \}$, has eigenvalues $\mu_1 \leqslant \mu_2 \leqslant \dots \leqslant \mu_{n-1}$.
			Then
			\[
				\lambda_1 \leqslant \mu_1 \leqslant \lambda_2 \leqslant \dots \leqslant \lambda_{n-1} \leqslant \mu_{n-1} \leqslant \lambda_n.
			\]
		\end{theorem}
		
		Next we prove a result about the top three coefficients of the characteristic polynomial of a principal submatrix of a Seidel matrix.
		
		\begin{lemma}\label{lem:coeff_of_interlacepoly}
		Let $S$ be a Seidel matrix of order $n$ and let $T = S[j]$ for some $j \in \{1,\dots,n\}$. 
		Suppose $S$ has minimal polynomial $m_S(x) = \sum_{i=0}^d a_i x^{d-i}$.
		% \[
% 		m_S(x)=x^m+ax^{m-1}+bx^{m-2}+\dots.
% 		\]
		Then 
		\[
		\chi_{T}(x)=\dfrac{\chi_S(x)}{m_S(x)}\sum_{i=0}^{d-1} b_i x^{d-1-i},
		\]
		where $b_0 = 1$, $b_1 = a_1$, $b_2 = a_2+n-1$, and $b_i \in \mathbb Z$ for $i \in \{3,\dots,d-1\}$.
		% \[
% 		f(x)=x^{m-1}+ax^{m-2}+(b+n-1)x^{m-3}+\dots.
% 		\]
		\end{lemma}
		\begin{proof}
			The proof is a straightforward consequence of Proposition~\ref{prop:submatrix}% and Theorem~\ref{thm:cauchyinterlace}
, using the fact that $\tr(S)=\tr(T)=0$, $\tr(S^2) = n(n-1)$, and $\tr(T^2) = (n-1)(n-2)$.
		\end{proof}
		
		The next result will be used repeatedly in this section.
		
		\begin{lemma}\label{lem:16elements}
			Let $\mathcal P$ denote the set of congruence classes of characteristic polynomials of Seidel matrices of order $49$ modulo $32\mathbb Z[x]$.
			Then $\mathcal P$ has $16$ elements.
		\end{lemma}
		\begin{proof}
			By Corollary~\ref{cor:countCharPolySeidelOdd}, the set $\mathcal P$ has cardinality at most $16$.
			Using a computer, it is straightforward to construct the set $\mathcal P$, and confirm that it has precisely $16$ elements.
		\end{proof}
		
		Now we can prove the first main result of this section.
		
		\begin{theorem}\label{thm:first_pol}
		There does not exist a Seidel matrix $S$ with the characteristic polynomial $\chi_S(x) = (x+5)^{33}(x-9)^{10}(x-11)^5(x^2-20x+95)$. 
		% $\{[-5]^{33},[9]^{12},[11]^4,[13]^1\}$.
		\end{theorem}
		\begin{proof}
			Suppose for a contradiction that $S$ is a Seidel matrix with the proposed spectrum, and let $\chi_S(x)$ be its characteristic polynomial. 
			Delete a row and (its corresponding) column of $S$ to form the Seidel matrix $S^\prime$.
			By Lemma~\ref{lem:coeff_of_interlacepoly}, we have
			\[
				\chi_{S^\prime}(x) = \frac{(x+5)^{33}(x-9)^{10}(x-11)^5(x^2-20x+95)}{(x+5)(x-9)(x-11)(x^2-20x+95)}(x^4-35x^3+443x^2-r_1x+r_0),
			\] for some integers $r_1$ and $r_0$.
			By Theorem~\ref{thm:cauchyinterlace}, the zeros of $\chi_{S^\prime}(x)$ must interlace those of $\chi_{S}(x)$.
			Hence the four zeros $\gamma_1 \leqslant \gamma_2 \leqslant \gamma_3 \leqslant \gamma_4$ of $f(x)=x^4-35x^3+443x^2-r_1x+r_0$ must satisfy $\gamma_1 \in [-5,10-\sqrt{5}]$, $\gamma_2 \in [10-\sqrt{5},9]$, $\gamma_3 \in [9,11]$, and $\gamma_4 \in [11,10+\sqrt{5}]$.
			Therefore, using the ideas in Section~\ref{sub:polynomial_enumeration_algorithm}, we find 286 possibilities for $f(x)$.
			
		Let $\mathcal P$ denote the set of congruence classes of characteristic polynomials of Seidel matrices of order $49$ modulo $32\mathbb Z[x]$.
		Using Lemma~\ref{lem:16elements}, we can produce all $16$ elements of $\mathcal P$.
		By checking against the congruence classes in $\mathcal P$, we reduce the number of possibilities for $f(x)$ down to three, that is, $f(x) = x^4-35x^3+443x^2-2381x+4516$, $f(x) = x^4-35x^3+443x^2-2369x+4392$, or $f(x) = x^4-35x^3+443x^2-2365x+4356$.  
		Each possibility for $f(x)$ corresponds to a putative characteristic polynomial for $S^\prime$.
		
		Using Proposition~\ref{prop:submatrix}, we find the row of the angle matrix of $S$ corresponding to each of the three putative characteristic polynomials for $S^\prime$.
		Using the ordering $\lambda_1 < \lambda_2 < \dots < \lambda_5$, where $\Lambda(S) = \{\lambda_1,\dots,\lambda_5\}$,
		we write these rows below.
			\begin{center}
			$\begin{array}{cccccc} 
			\mathbf{r}_1= & (2031/3080 & 2/55+\sqrt{5}/110 & 1/7 & 1/8 & 2/55-\sqrt{5}/110) \\
			\mathbf{r}_2= &(577/880 & 31/220+\sqrt{5}/44 & 0 & 1/16 & 31/220-\sqrt{5}/44) \\
			\mathbf{r}_3= &(36/55 & 19/110+\sqrt{5}/55 & 0 & 0 & 19/110-\sqrt{5}/55)
			\end{array}$ 
				\end{center}
		% Denote by $\mathbf{r}_1,\mathbf{r}_2,\mathbf{r}_{3}$ the above three putative rows of the angle matrix of $S$.
		If $S$ exists then, by Proposition~\ref{prop:sum_angles}, there must exist nonnegative integers $n_1,n_2,n_{3}$ such that $\sum_{i=1}^{3} n_i=50$ and $\sum_{i=1}^{3} n_i\mathbf{r}_i = (33,1,10,5,1)$.
		Since $\mathbf{r}_1$ is the only row to have a nonzero in the third column, we must have $n_1 = 70$, which is clearly impossible.
		Thus such a Seidel matrix $S$ cannot exist.	
		\end{proof}
		
		The proof of the next theorem is similar to that of Theorem~\ref{thm:first_pol}.
			
		\begin{theorem}\label{thm:second_pol}
		There does not exist a Seidel matrix $S$ with the characteristic polynomial $\chi_S(x) = (x+5)^{33}(x-9)^{12}(x-11)^4(x-13)$. 
		% $\{[-5]^{33},[9]^{12},[11]^4,[13]^1\}$.
		\end{theorem}
		\begin{proof}
		Suppose for a contradiction that $S$ is a Seidel matrix with the proposed spectrum, and let $\chi_S(x)$ be its characteristic polynomial. 
		Delete a row and (its corresponding) column of $S$ to form the Seidel matrix $S^\prime$.
		By Lemma~\ref{lem:coeff_of_interlacepoly}, we have
		\[
			\chi_{S^\prime}(x) = \frac{(x+5)^{33}(x-9)^{12}(x-11)^4(x-13)}{(x+5)(x-9)(x-11)(x-13)}(x^3-28x^2+243x-r),
		\] for some integer $r$.
		By Theorem~\ref{thm:cauchyinterlace}, the zeros of $\chi_{S^\prime}(x)$ must interlace those of $\chi_{S}(x)$.
		Hence the three zeros $\gamma_1 \leqslant \gamma_2 \leqslant \gamma_3$ of $x^3-28x^2+243x-r$ must satisfy $\gamma_1 \in [-5,9]$, $\gamma_2 \in [9,11]$, and $\gamma_3 \in [11,13]$.
		Therefore, we find that we must have $r \in \{616,\dots, 624\}$ (see Section~\ref{sub:polynomial_enumeration_algorithm}).
		
		Let $\mathcal P$ denote the set of congruence classes of characteristic polynomials of Seidel matrices of order $49$ modulo $32\mathbb Z[x]$.
		Using Lemma~\ref{lem:16elements}, we can produce all $16$ elements of $\mathcal P$.
		By checking against the congruence classes in $\mathcal P$, we find that $r = 616$.
		Hence we must have $\chi_{S^\prime}(x) = (x+5)^{32}(x-9)^{11}(x-11)^4(x^2-17x+56)$. 
		
		By Proposition \ref{prop:submatrix}, we find that, for all $j \in \{1,\dots,50\}$, the $j$th row of the angle matrix of $S$ is $(\alpha^2_{1j},\ldots,\alpha^2_{4j})=(83/126, 2/7, 0, 1/18)$. 
		However,
		\[
			\sum^{50}_{j=1}\alpha^2_{1j} = 50\cdot 83/126  \ne \dim \mathcal{E}(\lambda_1) = 33.
		\]
		This contradicts Proposition \ref{prop:sum_angles}.
		Thus such a Seidel matrix $S$ cannot exist.
		\end{proof}
		
		It remains to deal with the third possibility for the characteristic polynomial in Proposition~\ref{pro:onlyints}.
		We first establish a preparatory nonexistence result whose proof uses the same ideas as the previous two proofs. 
		
		\begin{proposition}\label{prop:kill_sub}
		There does not exist a Seidel matrix $S$ with characteristic polynomial $\chi_S(x) = (x+5)^{32}(x-7)(x-9)^{8}(x-11)^6(x^2-15x+48)$.
		 % \[\{[-5]^{32},[\frac{15-\sqrt{33}}{2}],[7],[9]^8,[\frac{15+\sqrt{33}}{2}],[11]^6\}.\]
		\end{proposition}
				\begin{proof}
				% We demonstrate this by showing that we cannot form the angle matrix from the possible choices of the interlacing polynomial of
		%
		% 		\[(x+5)^{32}(x-7)(x-9)^8(x-11)^6(x^2-15x+48).\]		%
				% Let $S$ be a Seidel matrix with characteristic polynomial
				% \[\chi_S(x)=(x+5)^{32}(x-7)(x-9)^8(x-11)^6(x^2-15x+48).\]
				Assume that $S$ is a Seidel matrix with the proposed spectrum, and let $\chi_S(x)$ be its characteristic polynomial. 
				Delete a row and (its corresponding) column of $S$ to form the Seidel matrix $S^\prime$.
				By Lemma~\ref{lem:coeff_of_interlacepoly}, we have
				\[
					\chi_{S^\prime}(x) = (x+5)^{31}(x-9)^7(x-11)^5f(x),
				\] where $f(x) = x^5-37x^4+530x^3+r_2x^2+r_1x+r_0$ for some integers $r_2$, $r_1$, and $r_0$.
				Let $\alpha_1 \leqslant \alpha_2$ be the two zeros of $x^2-15x+48$.
				Since the zeros of $\chi_{S^\prime}(x)$ must interlace those of $\chi_{S}(x)$, the five zeros $\gamma_1 \leqslant \dots \leqslant \gamma_5$ of $f(x)$ must satisfy $\gamma_1 \in [-5,\alpha_1]$, $\gamma_2 \in [\alpha_1,7]$, $\gamma_3 \in [7,9]$, $\gamma_4 \in [9,\alpha_2]$, and $\gamma_5 \in [\alpha_2,11]$.
				% Therefore we have $r \in \{612,\dots, 624\}$.
				% Given a principal submatrix of order 48, $S'$, its characteristic polynomial is
		%
		% 		\[\chi_{S'}(x)=(x+5)^{31}(x-9)^{7}(x-11)^{5}f(x),\]
		%
		% 		where, following from Proposition \ref{prop:edges}, $f(x)=(x^5-37x^3+530x^2+a_1x^2+a_2x+a_1)$ and it interlaces $(x+5)(x-7)(x-9)(x-11)(x^2-15x+48)$.

				% First, we assume that $f(x)$ is reducible.
		% 		Because of the above restrictions on the zeros of $f(x)$, it follows that either $f(x) = (x^2-15x+48)^2(x-7)$, $f(x) = (x-7)^2 g(x)$, or $f(x) = (x-9)^2 h(x)$, where $g(x)= x^3-23x^2+159x+g_0$ and $h(x)= x^3-19x^2+107x+h_0$ for some integers $g_0$ and $h_0$.
		% 		The three zeros of $g(x)$ (resp. $h(x)$) must respectively belong to the intervals $[-5,\alpha_1]$, $[9,\alpha_2]$, and $[\alpha_2,11]$ (resp. $[-5,\alpha_1]$, $[\alpha_1,7]$, and $[\alpha_2,11]$).
		% 		It follows that $g_0 \in [-297,-291]$ and $h_0 \in [-188,-182]$.
		% 		In total this gives us $1+7+7$ possibilities for the polynomial $f(x)$.
		% 		For each of these possibilities, we produce the polynomial $g(x) = (x+5)^{31}(x-9)^7(x-11)^5f(x)$ and we sieve out those polynomials for which $g(x-1)$ does not satisfy Corollary~\ref{cor:evencond}.
		% 		This leaves us with $2$ possible polynomials for $f(x)$:
		% 		\begin{center}
		% 		\vspace{-1em}
		% 		\resizebox{\columnwidth}{!}{%
		% 		\begin{tabular}{ c c }
		% 		$x^5 - 37x^4 + 530x^3 - 3650x^2 + 11997x - 14985;$ & $x^5 - 37x^4 + 530x^3 - 3650x^2 + 11949x - 14553.$
		% 			\end{tabular}
		% 			}
		% 		\end{center}
		
				% It remains to consider the case when $f(x)$ is irreducible.
				To find the possible values of $r_2$, $r_1$, and $r_0$, we apply the polynomial enumeration algorithm described in Section~\ref{sub:polynomial_enumeration_algorithm}. %(McKee's method)
				We find $22023$ possibilities for the vector $(r_2,r_1,r_0)$ and therefore there are $22023$ possibilities for $f(x)$.
				This computation ran in SageMath~\cite{sage} on a single core of an Intel Core i7 at 2.9 GHz and took 134 seconds.
		
				%For instance, $a_1\in \mathrm{Roots}(\Delta_{f''})\cap\mathbb{Z}$. Let us write $f_{a_1}$ for $f$ with a fixed value of $a_1$. Then, $a_2\in \mathrm{Roots}(\Delta_{f'_{a_1}})\cap\mathbb{Z}$ and $a_3\in \mathrm{Roots}(\Delta_{f_{a_1,a_2}})\cap\mathbb{Z}$. 
				% We find there are $2838$ possibilities for $f(x)$.
				From this list of $22023$ possibilities for $f(x)$, we produce the polynomial $g(x) = (x+5)^{31}(x-9)^7(x-11)^5f(x)$ and we sieve out those polynomials for which $g(x-1)$ does not satisfy Corollary~\ref{cor:evencond}. 
				This leaves us with $12$ possible polynomials for $f(x)$:
				\begin{center}
				\resizebox{\columnwidth}{!}{%
				\begin{tabular}{ c c }
				$x^5 - 37x^4 + 530x^3 - 3650x^2 + 11997x - 14985;$ & $x^5 - 37x^4 + 530x^3 - 3650x^2 + 11949x - 14553;$ \\
				$x^5 - 37x^4 + 530x^3 - 3666x^2 + 12237x - 15785;$ & $x^5 - 37x^4 + 530x^3 - 3658x^2 + 12109x - 15345;$\\ 
				$x^5 - 37x^4 + 530x^3 - 3658x^2 + 12109x - 15313;$ & $x^5 - 37x^4 + 530x^3 - 3658x^2 + 12109x - 15281;$\\ 
				$x^5 - 37x^4 + 530x^3 - 3658x^2 + 12093x - 15169;$ & $x^5 - 37x^4 + 530x^3 - 3650x^2 + 11965x - 14665;$\\ 
				$x^5 - 37x^4 + 530x^3 - 3650x^2 + 11981x - 14841;$ & $x^5 - 37x^4 + 530x^3 - 3650x^2 + 11965x - 14697;$\\
				$x^5 - 37x^4 + 530x^3 - 3642x^2 + 11837x - 14193;$ & $x^5 - 37x^4 + 530x^3 - 3642x^2 + 11821x - 14049.$
					\end{tabular}
					}
				\end{center}
		
				% Altogether we have $12$ possibilities for the characteristic polynomial $g(x)$ of $S^\prime$.
				Using Proposition~\ref{prop:submatrix}, we find the row of the angle matrix of $S$ corresponding to each of the $12$ putative characteristic polynomials for $S^\prime$.
				Using the ordering $\lambda_1 < \lambda_2 < \dots < \lambda_6$, where $\Lambda(S) = \{\lambda_1,\dots,\lambda_6\}$,
				we write these rows below.

					\begin{center}
						\resizebox{\columnwidth}{!}{%
						\begin{tabular}{ cc }
					$\begin{array}{cccccc} 
					(\frac{385}{592} & -\frac{7\sqrt{33}}{6512} + \frac{11}{592} & \frac{1}{8} & 0 & \frac{7\sqrt{33}}{6512} + \frac{11}{592} & \frac{3}{16}) \\
					(\frac{170}{259} & -\frac{31\sqrt{33}}{2442}+ \frac{17}{222} & 0 & \frac{4}{21} & \frac{31\sqrt{33}}{2442}+ \frac{17}{222} & 0) \\
					(\frac{8119}{12432} & -\frac{191\sqrt{33}}{19536}+ \frac{131}{1776} & \frac{1}{24} & \frac{2}{21} & \frac{191\sqrt{33}}{19536} + \frac{131}{1776} & \frac{1}{16}) \\
					(\frac{169}{259} & -\frac{19\sqrt{33}}{1221} + \frac{14}{111} & 0 & \frac{2}{21} & \frac{19\sqrt{33}}{1221} + \frac{14}{111} & 0) \\
					 (\frac{577}{888} & -\frac{67\sqrt{33}}{9768} + \frac{21}{296} & \frac{1}{12} & 0 & \frac{67\sqrt{33}}{9768} + \frac{21}{296} & \frac{1}{8}) \\
					    (\frac{31}{48} & \frac{\sqrt{33}}{528} + \frac{1}{16} & \frac{1}{24} & 0 & -\frac{\sqrt{33}}{528} + \frac{1}{16} & \frac{3}{16})
					\end{array}$ & 
					$\begin{array}{cccccc} 
					(\frac{24}{37} & -\frac{15\sqrt{33}}{814} + \frac{13}{74} & 0 & 0 & \frac{15\sqrt{33}}{814} + \frac{13}{74} & 0) \\
					(\frac{145}{222} & -\frac{19\sqrt{33}}{888} + \frac{39}{296} & \frac{1}{12} & 0 & \frac{19\sqrt{33}}{888} + \frac{39}{296} & 0)  \\ 
					(\frac{1353}{2072} & \frac{3\sqrt{33}}{1628} + \frac{7}{444} & 0 & \frac{4}{21} & -\frac{3\sqrt{33}}{1628} + \frac{7}{444} & \frac{1}{8}) \\
					(\frac{1345}{2072} & -\frac{5\sqrt{33}}{4884} + \frac{29}{444} & 0 & \frac{2}{21} & \frac{5\sqrt{33}}{4884} + \frac{29}{444} & \frac{1}{8}) \\
					(\frac{1153}{1776} & -\frac{247\sqrt{33}}{19536} + \frac{73}{592} & \frac{1}{24} & 0 & \frac{247\sqrt{33}}{19536} + \frac{73}{592} & \frac{1}{16})  \\
			
					 (\frac{191}{296} & -\frac{19\sqrt{33}}{4884} + \frac{17}{148} & 0  & 0  & \frac{19\sqrt{33}}{4884} + \frac{17}{148} & \frac{1}{8})
					\end{array}$
							\end{tabular}
							}
						\end{center}
				Denote by $\mathbf{r}_1,\dots,\mathbf{r}_{12}$ the above twelve putative rows of the angle matrix of $S$.
				If $S$ exists then, by Proposition~\ref{prop:sum_angles}, there must exist nonnegative integers $n_1,\dots,n_{12}$ such that $\sum_{i=1}^{12} n_i=49$ and $\sum_{i=1}^{12} n_i\mathbf{r}_i = (32,1,1,8,1,6)$.
				However, it is straightforward to verify (via linear programming) that there does not exist such $n_1,\dots,n_{12}$.
				% Finally, we check that no combination of the above angle forms the angle matrix that satisfies the conditions of Proposition \ref{prop:sum_angles}, that is $\sum_i n_i v_i=(32,1,1,8,1,6)$, where $v_i$ are the rows of angles, and $\sum_i n_i=49$, for $n_i\in \mathbb{N}$.
				\end{proof}
		
		Finally, using Proposition~\ref{prop:kill_sub}, we establish the nonexistence of a Seidel matrix having the remaining spectrum from Proposition~\ref{pro:onlyints}.
		% The proof of the next result is similar to the proof of Theorem~\ref{thm:first_pol}, but a bit of extra work is required.
% 		(This extra work is mostly done in Proposition~\ref{prop:kill_sub}.)
		
		\begin{theorem}
			\label{thm:third_pol}
		There does not exist a Seidel matrix with the characteristic polynomial $\chi_S(x) = (x+5)^{33}(x-7)(x-9)^{9}(x-11)^7$.
		% with the spectrum $\{[-5]^{33},[7],[9]^{9},[11]^7\}$.
		\end{theorem}
		\begin{proof}
		Assume that $S$ is a Seidel matrix with the proposed spectrum, and let $\chi_S(x)$ be its characteristic polynomial. 
		Delete a row and (its corresponding) column of $S$ to form the Seidel matrix $S^\prime$.
		By Lemma~\ref{lem:coeff_of_interlacepoly}, we have
		\[
			\chi_{S^\prime}(x) = \frac{(x+5)^{33}(x-7)(x-9)^9(x-11)^7}{(x+5)(x-7)(x-9)(x-11)}(x^3-22x^2+153x-r),
		\] for some integer $r$.
		By Theorem~\ref{thm:cauchyinterlace}, the zeros of $\chi_{S^\prime}(x)$ must interlace those of $\chi_{S}(x)$.
		Hence the three zeros $\gamma_1 \leqslant \gamma_2 \leqslant \gamma_3$ of $x^3-22x^2+153x-r$ must satisfy $\gamma_1 \in [-5,7]$, $\gamma_2 \in [7,9]$, and $\gamma_3 \in [9,11]$.
		Therefore, we find that we must have $r \in \{324,\dots, 336\}$ (see Section~\ref{sub:polynomial_enumeration_algorithm}).
		
		Let $\mathcal P$ denote the set of congruence classes of characteristic polynomials of Seidel matrices of order $49$ modulo $32\mathbb Z[x]$.
		Using Lemma~\ref{lem:16elements}, we can produce all $16$ elements of $\mathcal P$.
		By checking against the congruence classes in $\mathcal P$, we find that either $r = 324$ or $r=336$.
		Hence we must have either $\chi_{S^\prime}(x) = (x+5)^{32}(x-4)(x-9)^{10}(x-11)^6$ or $\chi_{S^\prime}(x) = (x+5)^{32}(x-7)(x-9)^{8}(x-11)^6(x^2-15x+48)$.
		
		Using Proposition~\ref{prop:submatrix}, we find that the angle matrix of $S$ has $42$ rows, each equal to $(37/56,0,3/14,1/8)$ and $8$ rows, each equal to $(21/32,1/8,0,7/32)$.
		Therefore, to show the nonexistence of $S$, it suffices to show the nonexistence of Seidel matrices having characteristic polynomial equal to either  $(x+5)^{32}(x-7)(x-9)^{8}(x-11)^6(x^2-15x+48)$ or $(x+5)^{32}(x-4)(x-9)^{10}(x-11)^6$.
		By Proposition~\ref{prop:kill_sub}, we have nonexistence of the former.
		\end{proof}
		
		\begin{remark}
			\label{rem:1749}
			Proposition~\ref{pro:onlyints}, together with Theorem~\ref{thm:first_pol}, Theorem~\ref{thm:second_pol}, and Theorem~\ref{thm:third_pol} shows that there cannot exist a Seidel matrix with smallest eigenvalue $-5$ having multiplicity $33$.
			Furthermore, $N(17) \leqslant 49$.
		\end{remark}
		
		\section{Concluding remarks} % (fold)
		\label{sec:concluding_remarks}
		
		The main ingredients for the improvement on the upper bound for $N(17)$ were obtained using the new restrictions on the characteristic polynomial of a Seidel matrix from Section~\ref{sec:a_modular_characterisation_of_the_characteristic_polynomial_of_a_seidel_matrix} and the computation of totally positive monic integer polynomials with trace minus degree at most $8$ in Section~\ref{sec:50lines}.
		The same approach can theoretically be applied to improve upper bounds for $N(d)$ for other $d$.
		However, the computation of the associated totally positive monic integer polynomials may become impractical.
		
		In our case, we found that the existence of $50$ equiangular lines in $\mathbb R^{17}$ was related to the existence of the polynomial $G(x)$, which had degree $17$ and trace $25$.
		Here it turned out that we could compute the all the relevant totally positive monic integer polynomials with trace minus degree at most $8$ in reasonable time.
		In general, for $n$ equiangular lines in $\mathbb R^d$, the analogous polynomial $G(x)$ will have degree $d$, but the trace may be much larger than $d$.
		When the trace of $G(x)$ is much larger than $d$, the computation of all possibilities for $G(x)$ will become much more computationally expensive.
		Hence, to successfully apply this method to improve the bounds for other $d$, we may require a more efficient method for computing totally positive monic integer polynomials.
		
		% As an example, we can consider the problem of finding the value of $N(14)$.
% 		Following Section~\ref{sub:from_equiangular_lies_to_seidel_matrices}, a system of $29$ equiangular lines in $\mathbb R^{14}$ would correspond to a Seidel matrix $S$ of order $29$ having smallest eigenvalue $-5$ with multiplicity $15$.
% 		Let $-5<\lambda_1\leqslant \lambda_2 \leqslant \dots \leqslant \lambda_{14}$ be the other eigenvalues $S$ then the analogues of \eqref{eqn:tr} and \eqref{eqn:tr2} are $\sum_{i=1}^{14} \lambda_i = 15 \cdot 5 = 75$ and $\sum_{i=1}^{14} \lambda_i^2 = 29\cdot 48 - 15 \cdot 5^2 = 437$.
% 		It follows that
% 		\[
% 			\sum_{i=1}^{14} (\lambda_i-6)^2 = 41.
% 		\]
% 		In this case, we cannot rule out the possibility that $6$ is an eigenvalue of $S$.
% 		We do know, however, that the eigenvalue $6$ has multiplicity at most $1$ \cite[Theorem 2.2]{GKMS16}.
% 		This means we could consider two cases: when $6$ is an eigenvalue of $S$ and when $6$ is not an eigenvalue of $S$.
% 		In the two cases we would want to find totally positive monic polynomials in $\mathbb Z[x]$ of degree $13$ (resp.\ $14$) and trace $41$.
% 		This would require a much more extensive computation then the one in this paper.
		However, there are other small dimensions $d$ for which the analogous $G(x)$ has small trace relative to its degree.
		As an example, we can consider the problem of finding the value of $N(19)$.
		The current upper bound for $N(19)$ is $75$ (see Table~\ref{tab:equi}).
		Following Section~\ref{sub:from_equiangular_lies_to_seidel_matrices}, a system of $75$ equiangular lines in $\mathbb R^{19}$ would correspond to a Seidel matrix $S$ of order $75$ having smallest eigenvalue $-5$ with multiplicity $56$.
		Let $-5<\lambda_1\leqslant \lambda_2 \leqslant \dots \leqslant \lambda_{19}$ be the other eigenvalues $S$.
		Then the analogues of \eqref{eqn:tr} and \eqref{eqn:tr2} are $\sum_{i=1}^{19} \lambda_i = 56 \cdot 5 = 280$ and $\sum_{i=1}^{19} \lambda_i^2 = 75\cdot 74 - 56 \cdot 5^2 = 4150$.
		It follows that 
		\[
			\sum_{i=1}^{19} (\lambda_i-15)^2 = 25.
		\]
		In this case, we cannot rule out the possibility that $15$ is an eigenvalue of $S$.
		% The case of $95$ equiangular lines in $\mathbb R^{20}$ is similar.
		This adds an extra complication to the problem that we will consider in a future paper.
		\end{document}